\documentclass[a4paper,reqno,11pt,oneside]{amsart}
\usepackage[left=2.7cm,right=2.7cm,top=3.5cm,
bottom=3.5cm]{geometry}
\usepackage[colorlinks=true,urlcolor=blue,
citecolor=red,linkcolor=blue,linktocpage,pdfpagelabels,
bookmarksnumbered,bookmarksopen]{hyperref}
\usepackage[english]{babel}
\usepackage{graphicx}
\usepackage{mathrsfs}
\usepackage{amssymb, amsmath, amsfonts, amsthm, mathtools}
\usepackage{xcolor}
\usepackage[shortlabels]{enumitem}
\usepackage{amsaddr}
\usepackage{latexsym}
\usepackage{etoolbox}
\usepackage{nccmath}

\makeatletter
\newtheorem*{rep@theorem}{\rep@title}
\newcommand{\newreptheorem}[2]{%
\newenvironment{rep#1}[1]{%
 \def\rep@title{#2 \ref{##1}}%
 \begin{rep@theorem}}%
 {\end{rep@theorem}}}
\makeatother

\makeatletter
\newcommand{\proofstep}[2]{%
  \par
  \addvspace{\medskipamount}%
  \noindent\emph{Step #1: #2}\par\nobreak
  \addvspace{\smallskipamount}%
  \@afterheading
}
\makeatother

\allowdisplaybreaks

\makeatletter
\@namedef{subjclassname@2020}{%
  \textup{2020} Mathematics Subject Classification}
\makeatother

\def\R{{\mathbb{R}}}
\def\N{{\mathbb{N}}}
\def\po{{\partial \Omega}}
\def\sphere{{\mathbb{S}^{N - 1}}}
\def\sphereplus{{\mathbb{S}_+^{N - 1}}}
\def\divergence{\mathrm{div }}
\def\Gz{{\Gamma_{0, \Omega}}}
\def\Gzp{{\Gamma_{0, \varphi}}}
\def\G{{\Gamma_{\Omega}}}
\def\Gp{{\Gamma_\varphi}}
\def\Op{{\Omega_\varphi}}
\def\OD{{\Omega_D}}
\def\cone{{\Sigma_D}}
\def\haus{{\mathcal{H}^{N - 1}}}
\def\sobspace{{H_0^1(\Omega \cup \Gamma_\Omega)}}
\def\capacity{{\text{cap}}}
\def\sobspcap{{H_0^1(\Omega; \cone)}}

\def\pc{{\partial\cone}}

\def\supp{{\text{supp }}}
\def\BTR{{B_{2R}(y_k)}}
\def\BFR{{B_{4R}(z_k)}}
\def\btr{{B_{\widetilde R}(0)}}
\def\ork{{\Omega_{R, k}}}
\def\tuk{{\widetilde{u}_k}}
\def\TUK{{v_k}}
\def\ctd{{C^{2, \alpha}(\overline{D})}}
\def\hu{{\widehat{u}}}
\def\F{{\mathcal{F}}}
\def\uphi{{u_\varphi}}
\def\vol{{\mathcal{V}}}
\def\M{{\mathcal{M}}}
\def\wtu{{\widetilde{u}}}
\def\wtv{{\widetilde{v}}}
\def\Gzreg{{\Gamma_{0, \Omega}^*}}
\def\GZsing{{\Gamma_{0, \Omega}^{\text{sing}}}}

\numberwithin{equation}{section}

\theoremstyle{plain}
\newtheorem{theorem}{Theorem}[section]
\newreptheorem{theorem}{Theorem}
\theoremstyle{plain}
\newtheorem{prop}[theorem]{Proposition}
\newreptheorem{prop}{Proposition}
\theoremstyle{plain}
\newtheorem{lemma}[theorem]{Lemma}
\newreptheorem{lemma}{Lemma}
\theoremstyle{plain}
\newtheorem{cor}[theorem]{Corollary}
\theoremstyle{definition}

\theoremstyle{definition}
\newtheorem{definition}[theorem]{Definition}
\theoremstyle{definition}
\newtheorem{remark}[theorem]{Remark}
\theoremstyle{definition}

\theoremstyle{plain}

\theoremstyle{definition}

\begin{document}

\renewcommand{\labelenumi}{\textit{(\roman{enumi})}}

\title[Non-radial minimizers for the first eigenvalue of the Laplacian in cones]{Non-radial minimizers for the first eigenvalue of the Laplacian in  cones and a related overdetermined problem}

\author[Danilo Gregorin Afonso]{Danilo Gregorin Afonso}
\address[Danilo Gregorin Afonso]{Università San Raffaele Roma \\ Via di Val Cannuta 247, 00166 Roma, Italy}
\email{danilo.afonso@uniroma5.it}



\date{\today}

\begin{abstract}
    In this work, we consider relative overdetermined problems for the first eigenfunction of the Laplacian for domains in cones, and the related question of minimizing the first eigenvalue among sets of a given fixed measure. By means of a shape derivative analysis, we show that the spherical sector is a critical shape and obtain a geometric condition on the cone for its stability/instability. By a concentration-compactness argument, we prove the existence of a minimizer, which moreover is bounded, open, connected, and whose relative boundary is regular almost everywhere. By another domain variation argument, we conclude that the minimizers admit a solution for the overdetermined problem.
\end{abstract}

\maketitle

\section{Introduction}
\label{sec:cone_intro}

Let $D \subset \sphere$ be a smooth domain (i.e. a connected open set) on the unit sphere $\sphere$ of $\R^N$, for $N \geq 3$. The cone spanned by $D$ is the set $\cone$ defined as
\begin{equation*}
    \cone \coloneqq \{x \in \R^N \ : \ x = r q, \ r > 0, \ q \in D\}.
\end{equation*}
For a domain $\Omega \subset \cone$, we set
\begin{equation*}
    \Gz \coloneqq \po \cap \cone, \quad \G \coloneqq (\po \cap \partial \cone)\setminus \partial \Gz,
\end{equation*}
and assume that $\haus(\G) > 0$, where $\haus$ denotes the $(N - 1)$-dimensional Hausdorff measure. 

Our aim is to study the following relative overdetermined problem:
\begin{equation}
    \label{eq:cone_overdet_pde}
    \left\{
    \begin{array}{rcll}
        - \Delta u & = & \lambda_1(\Omega) u & \quad \text{ in } \Omega \\
        u & = & 0 & \quad \text{ on } \Gz \\
        \displaystyle \mfrac{\partial u}{\partial \nu} & = & 0 & \quad \text{ on } \G \setminus \{O\} \\
        \mfrac{\partial u}{\partial \nu} & = & \text{constant} & \quad \text{ on } \Gz
    \end{array}
    \right.
    ,
\end{equation}
where $\lambda_1(\Omega)$ is the first eigenvalue of $- \Delta$ in $\Omega$ with the mixed boundary conditions $u = 0$ on $\Gz$ and $\mfrac{\partial u}{\partial \nu} = 0$ on $\G$ (we refer to Section \ref{sec:cone_prelim} for precise details). By a solution to \eqref{eq:cone_overdet_pde}, we will mean either the domain $\Omega$ or the corresponding solution $u$, with the precise meaning being clear from the context.  We will show that there exist domains $D \subset \sphere$ such that \eqref{eq:cone_overdet_pde} admits non-radial solutions, which are domains that minimize the shape functional $\Omega \mapsto \lambda_1(\Omega)$.

To clarify our motivations, let us begin by recalling that the problem with the free boundary in $\R^N$, which reads as
\begin{equation}
    \label{eq:cone_overdet_RN}
    \left\{
    \begin{array}{rcll}
        - \Delta u & = & \lambda_1(\Omega) u & \quad \text{ in } \Omega \\
        u & = & 0 & \quad \text{ on } \po \\
        \mfrac{\partial u}{\partial \nu} & = & \text{constant} & \quad \text{ on } \po
    \end{array}
    \right.
    ,
\end{equation}
with $\Omega$ bounded, admits a solution if and only if $\Omega$ is a ball. This can be proved, under suitable assumptions on the regularity of the boundary, by the method of moving planes, introduced in \cite{Alexandrov1962} in a geometric setting and later adapted to study symmetry for elliptic problems in \cite{Serrin1971}. We note that the positivity requirement for the method is satisfied, since, as is well known, the first Dirichlet eigenfunction of the Laplacian does not change sign.

A related question, which is also of great interest in itself, regards the shape optimization problem of minimizing the functional
\begin{equation} 
    \label{eq:cone_intro_functional}
    \Omega \mapsto \lambda_1(\Omega)
\end{equation}
among some suitable class of sets with a given fixed measure. It was first conjectured by Lord Rayleigh (\cite{Rayleigh1877}) that balls should minimize \eqref{eq:cone_intro_functional} in $\R^N$. The development of symmetrization theory allowed for the rigorous proof of the well-known Faber--Krahn inequality (\cite{LevitinMangoubiPolterovich2024published}, see also \cite{BrascoDePhilippisVelichkov2015} for a sharp quantitative form and \cite{BucurFreitas2017} for a free-boundary approach), which states that the ball is indeed the unique (up to negligible sets, in the capacity sense) minimizer for \eqref{eq:cone_intro_functional}. In $\R^N$, the relation between this shape optimization problem and \eqref{eq:cone_overdet_RN} comes from the well-known Hadamard formula for the shape derivative of simple eigenvalues of the Dirichlet-Laplacian (see \cite[Section 5.7]{HenrotPierre2018}), which in particular implies that any critical domain for \eqref{eq:cone_intro_functional}, if regular enough, admits a solution to \eqref{eq:cone_overdet_RN}. By a critical domain, we mean one where the shape derivative is zero with respect to all volume-preserving deformations.

Observe that for any cone $\Sigma_D$ spanned by a smooth domain $D \subset \sphere$, the spherical sector 
\begin{equation*}
    \Omega_D \coloneqq \Sigma_D \cap B_1(O)
\end{equation*}
admits a solution to the overdetermined problem \eqref{eq:cone_overdet_pde}, because the corresponding eigenfunction is simply the restriction of the first Dirichlet eigenfunction on the ball $B_1(O)$, which is radial, to the cone $\cone$. Hence problem \eqref{eq:cone_overdet_pde} is a natural generalization of \eqref{eq:cone_overdet_RN} to the relative setting, and it is natural to wonder whether or not the spherical sector is then the only solution.

There are (at least) two possible paths in the quest for new domains that admit a solution to \eqref{eq:cone_overdet_pde}: either to search for other critical shapes at a higher level (most likely of saddle type) for $\Omega \mapsto \lambda_1(\Omega)$ in a class of sufficiently smooth open sets, or to understand conditions under which $\Omega_D$ is not a minimizer, in which case a minimizer, if it exists (and is sufficiently smooth), will admit a solution for \eqref{eq:cone_overdet_pde}. 

It is important to point out that neither the method of moving planes nor symmetrization techniques are available in general cones. In fact, an isoperimetric inequality (which is a key ingredient for symmetrization theory) for convex cones was first shown in \cite{LionsPacella1990}, and later generalized to wider classes of cones in \cite{BaerFigalli2017, PacellaTralli2021}. This leads to Faber--Krahn-type inequalities in these classes, which, however, give only a partial answer to the problems we pose. 

To my knowledge, both the problem of deciding if the spherical sector is the only solution for \eqref{eq:cone_overdet_pde} and of finding a minimizer for \eqref{eq:cone_intro_functional} were still open in general cones.

Next, we shall briefly describe our results, outlining the strategy of the proofs. Our approach is inspired by ideas from \cite{BucurFreitas2017, IacopettiPacellaWeth2022, AfonsoIacopettiPacella2024Energypublished}. We refer to the respective sections for precise statements.

\begin{enumerate}[label=\roman*)]
    \item In Theorem \ref{thm:cone_stability}, we analyse the question of whether or not the spherical sector is a minimizer for \eqref{eq:cone_intro_functional}, because, if it is not, there is hope to find more exotic sets, the minimizers, which should admit a solution to \eqref{eq:cone_overdet_pde} (in view of Hadamard's formula). Theorem \ref{thm:cone_stability} gives a condition on the geometry of $D$ for which the spherical sector is/is not a local minimizer. This is done by computing the first and second shape derivatives of the eigenvalue in the class of polar graphs, following the approach of \cite{IacopettiPacellaWeth2022, AfonsoIacopettiPacella2024Energypublished}.

    \item In Theorem \ref{thm:cone_existence}, we give a condition under which the existence of minimizers for $\Omega \mapsto \lambda_1(\Omega)$ with a volume constraint is guaranteed. To this aim, we argue as in \cite{IacopettiPacellaWeth2022}, performing a delicate analysis of minimizing sequences, applying the concentration-compactness theorem, and making use of the fact that the cone ``flattens out" at infinity.

    \item Through a series of Lemmas and Propositions, we adapt arguments of \cite{Bucur2012, BriançonLamboley2009} to show that minimizers for \eqref{eq:cone_intro_functional} must be bounded, open, connected, and enjoy good regularity properties; see Section \ref{sec:cone_bound_conn_reg}.

    \item In Theorem \ref{thm:cone_overdet_minimizer}, we show that the minimizers for \eqref{eq:cone_intro_functional} admit a solution to the differential problem \eqref{eq:cone_overdet_pde} (in a weaker sense, which depends on the regularity of $\partial \Omega$), by computing ``localized" shape derivatives and exploiting the minimality.

    \item Theorem \ref{thm:cone_main} collects the previous results and shows the existence of non-radial minimizers for \eqref{eq:cone_intro_functional}, which admit non-radial solutions to \eqref{eq:cone_overdet_pde}.
\end{enumerate}

Analogous questions regarding the torsion problem (see Section \ref{sec:cone_prelim}) in cones and cylinders were considered in \cite{IacopettiPacellaWeth2022} (where the authors also study the relative isoperimetric problem), respectively in  \cite{CaldiroliIacopettiPacella2025}. General stability analysis for problems with autonomous nonlinearities in cones and cylinders is carried out in \cite{AfonsoIacopettiPacella2024Energypublished}. Other problems related to the eigenvalues of the Laplacian with mixed boundary conditions are considered in \cite{MazzoleniPellacciVerzini2020, ButtazzoVelichkov2016}.

An interesting remark is that the threshold for stability that we prove in Theorem \ref{thm:cone_stability} is precisely the same as that found by \cite{IacopettiPacellaWeth2022} for the torsion problem and \cite{AfonsoIacopettiPacella2024Energypublished} for nonlinear problems. It is expressed in terms of the first nonzero Neumann eigenvalue of the Laplace-Beltrami operator on the domain $D$ that spans the cone, which suggests that this eigenvalue, and not convexity, should be the correct parameter for the study of the isoperimetric inequality and related properties, as well as for an analysis of bifurcation from the radial objects.

Let us point out that the geometry of the cone plays a key role throughout the paper, not only as a motivation for its analogy with the ball. To begin with, we note that the class of polar graphs considered in Section \ref{sec:cone_stability} is quite well-behaved, which allows for a relatively easy computation of the second shape derivative of $\Omega \mapsto \lambda_1(\Omega)$. Next, in Section \ref{sec:cone_existence}, the crucial element in the proof of existence of a minimizer is the fact that the cone ``flattens out" at infinity, which allows us to compare a minimizing sequence with the eigenfunction of a half-ball. Finally, in Section \ref{sec:cone_bound_conn_reg}, the invariance of the cone by scaling allows us to consider equivalent minimization problems with volume penalization, for which the proof of qualitative properties is much easier.

This paper is organized as follows: Section \ref{sec:cone_prelim} collects some notation and preliminaries. In Section \ref{sec:cone_stability}, we study variations of \eqref{eq:cone_intro_functional} in the class of polar graphs, compute first and second shape derivatives with respect to these variations, and analyse the stability of the spherical sector as a critical shape under a volume constraint. Section \ref{sec:cone_existence} concerns the existence of a minimizer for \eqref{eq:cone_intro_functional}. Section \ref{sec:cone_bound_conn_reg} is devoted to the study of qualitative properties of the minimizers for \eqref{eq:cone_intro_functional}. Finally, in Section \ref{sec:cone_conclusion}, we prove the existence of non-radial minimizers for \eqref{eq:cone_intro_functional} and non-radial solutions for \eqref{eq:cone_overdet_pde}.

\section{Preliminaries}
\label{sec:cone_prelim}

\subsection{Notation}
Throughout the paper, we denote by $D$ a smooth domain (a connected open set) on $\sphere$, the unit sphere in $\R^N$, for $N \geq 3$, and denote by  $\cone$ the cone spanned by $D$:
\begin{equation*}
    \cone \coloneqq \{x \in \R^N \ : \ x = r q, \ r > 0, \ q \in D\}.
\end{equation*}
We assume that the domain $D$ is such that $\cone$ is a uniformly Lipschitz set.

For a domain $\Omega \subset \cone$, we set
\begin{equation*}
    \Gz \coloneqq \po \cap \cone, \quad \G \coloneqq (\po \cap \pc) \setminus \partial \Gz,
\end{equation*}
and assume that $\haus(\G) > 0$, where $\haus$ denotes the $(N - 1)$-dimensional Hausdorff measure.

\begin{remark}
    Since we chose the origin $O$ as the vertex of the cone, any radial regular function $u$ automatically satisfies $\mfrac{\partial u}{\partial \nu} = 0$ on $\pc \setminus \{O\}$.
\end{remark}

\subsection{A divergence theorem in sector-like domains}
Following the terminology introduced by \cite{PacellaTralli2020}, by a sector-like domain we mean an open set $\Omega \subset \cone$ such that $\Gz$ is a smooth $(N - 1)$-dimensional manifold, while $\partial \Gz = \partial \Gamma_\Omega \subset \pc \setminus\{O\}$ is a smooth $(N - 2)$-dimensional manifold.

\begin{lemma}[{\cite[Lemma 2.1]{PacellaTralli2020}}]
    \label{lemma:cone_divergence_theorem_sector_like_domains}
    Let $\Omega \subset \cone$ be a sector-like domain and let $F: \overline{\Omega} \to \R^N$ be a vector field such that
    \begin{align}
        & F \in C^1((\Omega \cup \Gz \cup \G) \setminus \{O\}) \cap L^2(\Omega), \nonumber \\ 
        & \divergence F \in L^1(\Omega). \nonumber
    \end{align}
    Then
    \begin{equation*}
        \int_\Omega \divergence F \ dx = \int_\Gz F \cdot \nu \ d\sigma_{\Gz} + \int_{\G \setminus \{O\}} F \cdot \nu \ d\sigma_{\pc}.
    \end{equation*}
\end{lemma}

\subsection{Capacity and quasi-open sets}
\begin{definition}
    The capacity of a set $E \subseteq \R^N$ is defined as
    \begin{equation*}
        \capacity(E) \coloneqq \inf\{\|v\|^2_{H^1(\R^N)} \ : \ v \geq 1 \text{ in a neighbourhood of } E\}.
    \end{equation*}
\end{definition}

\begin{definition}
    A set $U \subset \R^N$ is said to be quasi-open if for every $\varepsilon > 0$ there exists an open set $V_\varepsilon$ such that $U \cup V_\varepsilon$ is open and $\capacity(V_\varepsilon) \leq \varepsilon$.
\end{definition}

For any quasi-open set $\Omega \subset \cone$, we define the Sobolev space
\begin{equation*}
    \sobspcap \coloneqq \{u \in H^1(\cone) \ : \ u = 0 \text{ q.e. in } \cone \setminus \Omega\}.
\end{equation*}
Here, q.e. stands for quasi-everywhere, that is, up to sets of zero capacity. It can be shown that this space is the classical Sobolev space $H_0^1(\Omega \cup \Gamma_\Omega)$ if $\Omega$ is open. We refer to \cite{HenrotPierre2018} for more details on these questions.

\subsection{Eigenvalue problems with mixed boundary conditions}

For a bounded domain $\Omega \subset \cone$ with a smooth relative boundary $\Gz$, we consider the following eigenvalue problem with mixed Dirichlet-Neumann boundary conditions, also known as Zaremba problem (\cite{LevitinMangoubiPolterovich2024published})
\begin{equation}
    \label{eq:cone_eigenvalue_problem}
    \left\{
    \begin{array}{rcll}
        - \Delta u & = & \lambda u & \quad \text{ in } \Omega \\
        u & = & 0 & \quad \text{ on } \Gz \\
        \displaystyle \mfrac{\partial u}{\partial \nu} & = & 0 & \quad \text{ on } \G \setminus \{O\}
    \end{array}
    \right.
\end{equation}
It can be shown that this problem has a variational formulation in the space $\sobspace$, which is the subspace of $H^1(\Omega)$ of functions whose trace vanishes on $\Gz$. Moreover, it can be shown that \eqref{eq:cone_eigenvalue_problem} inherits many features of its ``pure Dirichlet" counterpart; see \cite[Section 1.4]{DamascelliPacella2019}. In particular, under our assumptions,
\begin{enumerate}[label=(\roman*)]
    \item there exists a sequence of eigenvalues $\{\lambda_j\}_{j \in \mathbb{N}}$, counted with multiplicity, such that $\lambda_j \to + \infty$ as $j \to \infty$, and the corresponding eigenfunctions form an orthonormal basis of the space $L^2(\Omega)$;

    \item the eigenvalues admit a min-max characterization. In particular, 
    \begin{equation*}
        \lambda_1(\Omega) = \inf_{\substack{v \in \sobspace \\ v \neq 0}} \frac{\displaystyle \int_\Omega |\nabla v|^2 \ dx}{\displaystyle \int_\Omega v^2 \ dx } = \inf_{\substack{v \in \sobspace \\ \|v\|_2 = 1}} \int_\Omega |\nabla v|^2 \ dx > 0;
    \end{equation*}

    \item the first eigenvalue is simple and the corresponding eigenfunction does not change sign;

    \item the eigenvalues are non-increasing with respect to inclusion: if $\Omega' \subset \Omega$, then $\lambda_j(\Omega') \geq \lambda_j(\Omega)$.
\end{enumerate}

These notions can be generalized from the class of bounded domains to the class of quasi-open subsets of $\cone$ with finite measure. In fact, such an extension is crucial for the minimization argument carried out in Section \ref{sec:cone_existence}, since quasi-open sets are, roughly speaking, level sets of Sobolev functions and our minimizer will be precisely the support of a positive function $u \in H^1(\cone)$, which is not known to be open \textit{a priori}. 

For a quasi-open set $\Omega \subset \cone$ of finite measure, the eigenvalue problem \eqref{eq:cone_eigenvalue_problem} is to be interpreted in the variational form: a solution is a pair $(u, \lambda)$, with $u \in \sobspcap$, such that
\begin{equation*}
    \int_\cone \nabla u \nabla v \ dx = \lambda \int_\cone u v \ dx \quad \forall v \in \sobspcap.
\end{equation*}

Since the domain $D$ that spans the cone is smooth and connected, $\cone$ is smooth except at the vertex (where, obviously, it satisfies the cone condition). Then it can be shown that for any $\Omega$ of finite measure, the inclusion $\sobspcap \hookrightarrow L^2(\cone)$ is compact, and one is then able to develop a spectral theory akin to the case of bounded open sets; see \cite{ButtazzoVelichkov2016, Velichkov2015, HenrotPierre2018} for the general theory of partial differential equations in the class of quasi-open sets. In particular, for every quasi-open set $\Omega \subset \cone$ of finite measure, we consider
\begin{equation*}
    \lambda_1(\Omega) = \inf_{\substack{v \in \sobspcap \\ v \neq 0}} \frac{\displaystyle \int_\Omega |\nabla v|^2 \ dx}{\displaystyle \int_\Omega v^2 \ dx } = \inf_{\substack{v \in \sobspcap \\ \|v\|_2 = 1}} \int_\Omega |\nabla v|^2 \ dx > 0.
\end{equation*}
We denote the corresponding nonnegative $L^2$-normalized eigenfunction by $u_\Omega$.

\subsection{Torsional energy}
We now recall some facts from \cite{IacopettiPacellaWeth2022} regarding the relative torsional energy problem for domains in $\cone$. If $\Omega$ is an open set with smooth relative boundary $\Gz$, then the so-called torsion function of $\Omega$ is the solution $w_\Omega \in H_0^1(\Omega \cup \G)$ for
\begin{equation*}
    \left\{
    \begin{array}{rcll}
        - \Delta w_\Omega & = & 1 & \quad \text{ in } \Omega \\
        w_\Omega & = & 0 & \quad \text{ on } \Gz \\
        \mfrac{\partial w_\Omega}{\partial \nu} & = & 0 & \quad \text{ on } \G
    \end{array}
    \right.
    .
\end{equation*}
It can be shown that $w_\Omega$ is the unique minimizer of the functional
\begin{equation*}
    J(v) = \frac{1}{2} \int_\Omega |\nabla v|^2 \ dx - \int_\Omega v \ dx, \quad v \in \sobspace,
\end{equation*}
which thus allows us to define the shape functional
\begin{equation*}
    E(\Omega) = J(w_\Omega).
\end{equation*}

As for the eigenvalue problem, we can weaken the requirements and consider the torsion function $w_\Omega$ of a quasi-open set $\Omega \subset \cone$ by passing to the functional space $\sobspcap$.

It can also be shown (see \cite{Bucur2012, Velichkov2015}) that the function
\begin{equation}
    \label{eq:cone_def_d_gamma}
    d_\gamma(\Omega, \omega) \coloneqq \int_{\cone} |w_\Omega - w_\omega| \ dx, \quad \Omega, \omega \subset \cone \text{ quasi-open},
\end{equation}
is a distance on the class of quasi-open sets of $\cone$ of finite measure with the topology of $\gamma$-convergence (see \cite{HenrotPierre2018, ButtazzoVelichkov2016}).

\section{Stability analysis of the spherical sector}
\label{sec:cone_stability}

In this section, we follow the ideas of \cite{IacopettiPacellaWeth2022, AfonsoIacopettiPacella2024Energypublished}. Thus we focus on domains which are strictly star-shaped with respect to the vertex of the cone, which we assume to be the origin $O$ of $\R^N$. Such domains can be parametrized by functions in $\ctd$ as follows. For $\varphi \in \ctd$, we set
\begin{equation}
    \label{eq:cone_def_Omega_varphi}
    \Omega_\varphi \coloneqq \{x \in \cone \ : \ x = r q, \ 0 < r < e^{\varphi(q)}, q \in D\}.
\end{equation}
For simplicity of notation, we set
\begin{equation*}
    \Gzp \coloneqq \Gamma_{0, \Omega_\varphi}, \quad \Gp \coloneqq \Gamma_{\Omega_\varphi}.
\end{equation*}
In this way, we can consider the functional $\lambda_1: \ctd \to \R$ given by
\begin{equation*}
    \lambda_1(\varphi) = \lambda_1(\Omega_\varphi), \quad \varphi \in \ctd.
\end{equation*}
To $\lambda_1(\varphi)$ corresponds a positive $L^2$-normalized eigenfunction $u_\varphi$, which is a classical solution to
\begin{equation*}
    \left\{
    \begin{array}{rcll}
        - \Delta u_\varphi & = & \lambda_1(\varphi) u_\varphi & \quad \text{ in } \Omega_\varphi \\
        u_\varphi & = & 0 & \quad \text{ on } \Gzp \\
        \mfrac{\partial u_\varphi}{\partial \nu} & = & 0 & \quad \text{ on } \Gp \setminus \{O\}.
    \end{array}
    \right.
\end{equation*}

\begin{remark}
    Note that by passing through the exponential function instead of considering the graph of the function $\varphi$ as the relative boundary of $\Omega_\varphi$ directly, we have all the space $\ctd$ available as possible variations and do not have to worry about $\varphi$ being too negative. This will be useful for our aim of computing first- and second-order derivatives of $\lambda_1(\varphi)$.
\end{remark}

Let us now describe the domain variations that we consider. For $\eta \in \ctd$ and $t \in (- \delta, \delta)$, with $\delta > 0$ fixed, we consider perturbed domains of the type 
\begin{equation*}
    \Omega_{\varphi + t\eta} \subset \cone, \ t \in (- \delta, \delta)
\end{equation*}
We note that such variations can be described by a one-parameter family of diffeomorphisms as follows. Let 
\begin{equation*}
    \xi: (t, x) \in (- \delta, \delta) \times \cone \mapsto e^{t \eta\left(\frac{x}{|x|} \right)} x \in \cone
\end{equation*}
It is easily verified that
\begin{equation*}
    \xi|_{\Omega_\varphi}(t, \cdot): \Omega_\varphi \to \Omega_{\varphi + t\eta}
\end{equation*}
is a diffeomorphism, whose inverse map $\left(\xi|_{\Omega_{\varphi}} \right)^{-1}: \Omega_{\varphi + t\eta} \to \Omega_\varphi$ is given by
\begin{equation*}
    \left(\xi|_{\Omega_\varphi} \right)^{-1}(x) = e^{- t\eta\left(\frac{x}{|x|} \right)}x = \xi(-t, x).
\end{equation*}
In addition, it immediately follows from the definition that
\begin{equation*}
    \xi(t, x) \in \pc \setminus \{O\} \quad \forall (t, x) \in (- \delta, \delta) \times (pc \setminus \{O\}).
\end{equation*}
Then $\xi$ is the flow corresponding to the vector field $V: \cone \to \R^N$ given by
\begin{equation}
    \label{eq:cone_vector_field}
    V_\eta(x) = \eta\left(\frac{x}{|x|} \right)x.
\end{equation}
Indeed: $\xi(0, x) = x$ for all $x \in \cone$, 
\begin{equation*}
    \frac{d \xi}{dt}(t, x) = e^{t \eta\left(\frac{x}{|x|} \right)} \eta\left( \frac{x}{|x|}\right)x = V_\eta(\xi(t, x)),
\end{equation*}
and $\{\Omega_{\varphi + t\eta}\}_{t \in (-\delta, \delta)}$ is a family of deformations of $\Omega_\varphi$ corresponding to the vector field $V_\eta$, i.e. to the one-parameter family of diffeomorphisms 
\begin{equation*}
\xi_t(\cdot) \coloneqq \xi(t, \cdot).
\end{equation*}

\begin{remark}
    \label{rem:cone_normal_dsigma}
    The outer unit normal vector to $\Gzp$ is given by
    \begin{equation*}
        \nu(x) = \frac{\displaystyle \frac{x}{|x|} - \nabla_{\sphere} \varphi\left(\frac{x}{|x|}\right)}{\displaystyle \sqrt{1 + \left|\nabla_\sphere \varphi\left(\frac{x}{|x|} \right)\right|^2}}.
    \end{equation*} 
    Moreover, the area element in $\Gzp$ is given by
    \begin{equation*}
        d \sigma_{\Gzp} = e^{(N - 1) \varphi} \sqrt{1 + |\nabla_\sphere \varphi|^2} d\sigma,
    \end{equation*}
    where $d\sigma$ is the area element on $D$; see \cite{IacopettiPacellaWeth2022}.
\end{remark}

\subsection{Shape derivatives of $\lambda_1$}

Our first main result of the section is the computation of the first derivative of $\lambda_1(\varphi)$ with respect to a variation $\eta \in \ctd$.

\begin{prop}
    \label{prop:cone_first_derivative}
    Let $\varphi \in \ctd$. For any $\eta \in \ctd$, there exists $\delta > 0$ such that the map
    \begin{equation*}
        t \in (- \delta, \delta) \mapsto (u_{\varphi + t\eta} \circ \xi_t, \lambda_1(\varphi + t\eta)) \in L^2(\cone) \times \R
    \end{equation*}
    is of class $C^\infty$ in $(- \delta, \delta)$ and $\wtu_\eta = \left.\mfrac{d}{dt} (u_{\varphi + t\eta}) \right|_{t = 0} \in H^1(\Omega_\varphi)$ is the unique solution of
    \begin{equation}
        \label{eq:cone_pde_u_prime_general}
        \left\{
        \begin{array}{rcll}
            - \Delta \wtu_\eta & = & \lambda_1(\varphi) \wtu_\eta + \lambda' u_\varphi & \quad \text{ in } \Omega_\varphi \\
            \wtu_\eta & = & \displaystyle - \mfrac{\partial u_\varphi}{\partial \nu} (V_\eta \cdot \nu) & \quad \text{ on } \Gzp \\
            \displaystyle \mfrac{\partial \wtu_\eta}{\partial \nu} & = & 0 & \quad \text{ on } \Gp \setminus \{O\} \\
            \displaystyle \int_{\Omega_\varphi} \wtu_\eta u_\varphi \ dx & = & 0 &
        \end{array}
        \right.
        ,
    \end{equation}    
    where $\lambda' \coloneqq \left. \frac{d \lambda_t}{dt}\right|_{t = 0}$ is given by
    \begin{equation}
        \label{eq:cone_lambda_prime}
        \lambda'(\varphi)[\eta] = \left. \frac{d \lambda_t}{dt}\right|_{t = 0} = - \int_D e^{N\varphi(q)} \eta(q) \left(\frac{\partial u_\varphi}{\partial \nu}(e^{\varphi(q)}q)\right)^2 \ d\sigma.
    \end{equation}
\end{prop}

\begin{proof}
    Let $\delta > 0$ be a positive number to be chosen later. For every $t \in (- \delta, \delta)$ sufficiently small, the eigenvalue $\lambda_t \coloneqq \lambda_1(\varphi + t\eta)$ and the corresponding eigenfunction $u_t \coloneqq u_{\Omega_{\varphi + t\eta}}$ are well-defined in the domain $\Omega_t \coloneqq \Omega_{\varphi + t\eta}$. Moreover, the map
    \begin{equation*}
        t \in (- \delta, \delta) \mapsto (u_t, \lambda_t) \in H_0^1(\Omega_t \cup \Gamma_t) \times \R
    \end{equation*}
    where for brevity we have set $\Gamma_t \coloneqq \Gamma_{\varphi + t\eta}$, is continuous (\cite[Chapter 4]{HenrotPierre2018}). The variational formulation of the eigenvalue problem in $\Omega_t$ is
    \begin{align}
        & u_t \in H_0^1(\Omega_t \cup \Gamma_t), \nonumber \\
        & \int_{\Omega_t} u_t^2\ dx = 1, \nonumber \\
        & \int_{\Omega_t} \nabla u_t \nabla v \ dx = \lambda_t \int_{\Omega_t} u_t v \ dx \quad \forall v \in H_0^1(\Omega_t \cup \Gamma_t). \nonumber
    \end{align}

    As usual in the theory of shape derivatives, the idea is to transport this equation to the fixed domain $\Omega_\varphi$ and then to apply the implicit function theorem (see \cite{HenrotPierre2018, SokolowskiZolesio1992}). This can be done by virtue of the fact that
    \begin{equation*}
        H_0^1(\Omega_\varphi \cup \Gp) = \{v \circ \xi_t \ : \ v \in H_0^1(\Omega_t \cup \Gamma_t)\}. 
    \end{equation*}
    Thus, under the change of variables given by $\xi_t$, the transported function
    \begin{equation*}
        \hu_t \coloneqq u_t \circ \xi_t
    \end{equation*}
    satisfies
    \begin{align}
        & \hu_t \in H_0^1(\Omega_\varphi \cup \Gp), \nonumber \\
        & \int_{\Omega_\varphi} \hu_t^2 J_t \ dx = 1, \nonumber \\ 
        & \int_{\Omega_\varphi} (A_t \nabla \hu_t) \nabla w \ dx = \lambda_t \int_{\Omega_\varphi} \hu_t w J_t \ dx \quad \forall w \in H_0^1(\Omega_\varphi \cup \Gp), \nonumber
    \end{align}
    where, denoting by $D\xi_t$ the Jacobian matrix of $\xi_t$,  
    \begin{equation*}
        J_t = \det D\xi_t,
    \end{equation*}
    which is positive for $t$ small because $\xi_t$ is close to the identity, and 
    \begin{equation*}
        A_t = J_t(D\xi_t^{-1}) (D \xi_t^{-1})^T.
    \end{equation*}
    Note that $J_t$ and $A_t$ are in $L^\infty$, even if $\xi_t$ and $V$ are not differentiable at the vertex of the cone.

    Next, we define the operator
    \begin{equation*}
        \F : (- \delta, \delta) \times H_0^1(\Omega_\varphi \cup \Gp) \times \R \to H^{-1}(\Omega_\varphi \cup \Gp) \times \R
    \end{equation*}
    where $H^{-1}(\Omega_\varphi \cup \Gp)$ denotes the dual space of $H_0^1(\Omega_\varphi \cup \Gp)$, by
    \begin{equation*}
        \F(t, v, \lambda) = \left(- \divergence(A_t \nabla v) - \lambda v J_t, \int_{\Omega_\varphi} v^2 J_t \ dx - 1\right).
    \end{equation*}
    It is a standard fact from the theory of shape derivatives that the map $\F$ is of class $C^\infty$ (see, e.g., \cite[Page 207]{HenrotPierre2018}).
    
    Note that 
    \begin{equation*}
        \F(t, \hu_t, \lambda_t) = 0 \quad \forall t \in (-\delta, \delta).
    \end{equation*}
    We aim to apply the implicit function theorem to obtain the differentiability of $\hu_t$, and hence of $u_t$, at $t = 0$. To this end, let us note that
    \begin{equation*}
        \partial_{v, \lambda}\F(0, u_\varphi, \lambda_1(\varphi)) [v, \lambda] = \left(- \Delta v - \lambda u_\varphi - \lambda_1(\varphi) v, 2 \int_\Op u_\varphi v \ dx \right), \quad (v, \lambda) \in H_0^1(\Omega_\varphi \cup \Gp) \times \R.
    \end{equation*}
    It is also standard to show that 
    \begin{equation*}
        \partial_{v, \lambda}\F(0, u_\varphi, \lambda_1(\varphi)) : H_0^1(\Omega_\varphi \cup \Gp) \times \R \to H^{-1}(\Omega_\varphi \cup \Gp) \times \R
    \end{equation*}
    is an isomorphism; the proof goes as in \cite[Lemma 5.7.3]{HenrotPierre2018}, with small modifications regarding the functional spaces (see also \cite{LamboleySicbaldi2014}). Hence, by the implicit function theorem, up to taking a smaller $\delta > 0$, it holds that the map
    \begin{equation*}
        t \in (- \delta, \delta) \mapsto (\hu_t, \lambda_t)
    \end{equation*}
    is of class $C^\infty$. In addition, since $u_t = \hu_t \circ \xi_t^{-1}$, then small modifications of \cite[Lemma 5.3.3] {HenrotPierre2018} yield that 
    \begin{equation*} 
    t \in (- \delta, \delta) \mapsto u_t \in L^2(\cone)
    \end{equation*}
    is differentiable at $t = 0$ and
    \begin{equation}
        \label{eq:cone_u_prime}
        \wtu_\eta = \left.\mfrac{d}{dt}\hu_t\right|_{t = 0} - \nabla u_\varphi \cdot V_\eta \in H^1(\Omega_\varphi)
    \end{equation}
    where $\wtu_\eta = \left.\mfrac{d}{dt}\right|_{t = 0} u_t$ and $V_\eta$ is the vector field corresponding to the deformations $\xi_t$ (see \eqref{eq:cone_vector_field}).

    Let us now characterize $\wtu_\eta$ as the solution of a differential problem. Since $\nabla u_\varphi = \mfrac{\partial u_\varphi}{\partial \nu} \nu$ on $\Gzp$ and $\left.\mfrac{d}{dt}\hu_t\right|_{t = 0} \in H_0^1(\Omega_\varphi \cup \Gp)$, then
    \begin{equation*}
        \wtu_\eta = - \frac{\partial u_\varphi}{\partial \nu} (V \cdot \nu) \quad \text{ on } \Gzp.
    \end{equation*}
    Next, arguing as in the proof of \cite[Theorem 5.3.1]{HenrotPierre2018}, namely taking $v \in C_c^\infty(\Omega_\varphi)$, integrating by parts and applying Hadamard's formula (\cite[Section 5.2]{HenrotPierre2018}), it follows that
    \begin{equation*}
        \int_{\Omega_\varphi} (- \Delta \wtu_\eta - \lambda_1(\varphi)\wtu_\eta - \lambda' u_\varphi) v \ dx = 0. 
    \end{equation*}
    Since $v \in C_c^\infty(\Omega_\varphi)$ is arbitrary, we conclude that $\wtu_\eta$ satisfies, in the weak sense (and therefore, by standard elliptic regularity theory, also in the classical sense)
    \begin{equation*}
        - \Delta \wtu_\eta = \lambda_1(\varphi) \wtu_\eta + \lambda' u_\varphi \quad \text{ in } \Omega_\varphi.
    \end{equation*}
    Also, by Hadamard's formula, differentiating the relation
    \begin{equation*}
        \int_{\Omega_t} u_t^2 \ dx = 1 \quad \forall t \in (- \delta, \delta)
    \end{equation*}
    at $t = 0$ yields
    \begin{equation*}
        2 \int_{\Omega_\varphi} u_\varphi \wtu_\eta \ dx + \int_{\Gp} u_\varphi^2 (V \cdot \nu) \ d\sigma = 0.
    \end{equation*}
    Since $V \perp \nu$ on $\pc$, then we obtain the orthogonality constraint
    \begin{equation*}
        \int_{\Omega_\varphi} u_\varphi \wtu_\eta \ dx = 0.
    \end{equation*}
    Finally, differentiating the relation 
    \begin{equation*}
        \nabla u_t(\xi_t(x)) \cdot \nu_{\xi_t(x)} = 0 \quad \forall t \in (- \delta, \delta),
    \end{equation*}
    which holds because $V$ is tangent to $\pc$ on $\pc$, and arguing as in \cite[Section 5.5]{HenrotPierre2018}, we obtain
    \begin{equation*}
        \frac{\partial \wtu_\eta}{\partial \nu} = 0 \quad \text{ on } \Gp \setminus \{O\}.
    \end{equation*}
    Summarizing, $\wtu_\eta \in H^1(\Omega_\varphi)$ satisfies \eqref{eq:cone_pde_u_prime_general}.

    The differential problem \eqref{eq:cone_pde_u_prime_general} allows us to obtain a general expression for $\lambda'$ as follows. Multiplying the differential equation by $u_\varphi$, taking into account that $\int_{\Omega_\varphi} \wtu_\eta u_\varphi \ dx = 0$, and integrating by parts (making use of Lemma \ref{lemma:cone_divergence_theorem_sector_like_domains}), we obtain:
    \begin{align}
        \lambda' \underbrace{\int_\Op u_\varphi^2 \ dx}_{= 1}
        & = \int_\Op \nabla \wtu_\eta \nabla u_\varphi \ dx - \underbrace{\int_\Op \lambda_1(\varphi) \wtu_\eta u_\varphi \ dx}_{= 0} \nonumber \\
        & = \underbrace{\int_\Op \lambda_1(\varphi) u_\varphi \wtu_\eta \ dx}_{= 0} + \int_{\Gzp} \wtu_\eta \frac{\partial u_\varphi}{\partial \nu} \ d\sigma \nonumber \\
        & = - \int_{\Gzp} \left(\frac{\partial u_\varphi}{\partial \nu} \right)^2 (V \cdot \nu) \ d\sigma. \nonumber 
    \end{align}
    Formula \eqref{eq:cone_lambda_prime} then follows immediately, in view of Remark \ref{rem:cone_normal_dsigma} and since
    \begin{equation*}
        V \cdot \nu = \frac{|x|}{\displaystyle \sqrt{1 + \left|\nabla_\sphere \varphi \left(\frac{x}{|x|} \right)\right|^2}} \eta\left(\frac{x}{|x|} \right) \quad \text{ on } \Gzp.
    \end{equation*}
    The proof is complete.
\end{proof}

We now provide a general formula for the second variation of the functional $\varphi \mapsto \lambda_1(\varphi)$. To this aim, we make the further assumption that the domain $\Omega_\varphi$ touches $\pc$ orthogonally, in such a way that $u$ has the required regularity near $\partial \Gzp$.

\begin{prop}
    \label{prop:cone_second_derivative}
    Let $\varphi \in \ctd$ be such that $\mfrac{\partial \varphi}{\partial \nu_{\partial D}} = 0$ on $\partial D$. For any $\eta, \zeta \in \ctd$, it holds that
    \begin{align}
        \lambda''(\varphi)[\eta, \zeta]
        & \coloneqq \left.\frac{d}{ds}\left(\lambda'(\varphi + s\zeta)[\eta] \right)\right|_{s = 0} \nonumber \\
        & = - N \int_D e^{N \varphi} \eta \zeta \left(\frac{\partial \uphi}{\partial \nu}(e^\varphi q) \right)^2 \ d\sigma \nonumber \\
        & \quad -2 \int_D e^{N \varphi} \eta \frac{\partial \uphi}{\partial \nu}(e^\varphi q) \frac{\partial \wtu_\zeta}{\partial \nu}(e^\varphi q) \ d\sigma \nonumber \\
        & \quad - 2 \int_D e^{N \varphi} \eta \zeta \frac{\partial \uphi}{\partial \nu}(e^\varphi q) [D^2\uphi(e^\varphi q) (e^\varphi q)] \cdot \nu \ d\sigma \nonumber \\
        & \quad + 2 \int_D e^{N \varphi} \eta \frac{\partial \uphi}{\partial \nu}(e^\varphi q) \frac{\nabla \uphi(e^\varphi q) \nabla_\sphere \zeta}{\sqrt{1 + |\nabla_\sphere\varphi|^2}} \ d\sigma \nonumber \\
        & \quad + 2 \int_D e^{N \varphi} \eta \left(\frac{\partial \uphi}{\partial \nu}(e^\varphi q) \right)^2 \frac{\nabla_\sphere \varphi \nabla_\sphere \zeta}{(1 + |\nabla_\sphere \varphi|^2)} \ d\sigma, \label{eq:cone_lambda_primeprime}
    \end{align}
    where 
    \begin{equation*}
        \wtu_\zeta = \left.\frac{d}{ds}(u_{\varphi + s\zeta}) \right|_{s = 0}
    \end{equation*}
    is the solution of \eqref{eq:cone_pde_u_prime_general} with 
    \begin{equation*}
        V_\zeta(x) = \zeta\left(\frac{x}{|x|} \right)x.
    \end{equation*}
\end{prop}

\begin{proof}
    By definition and by Proposition \ref{prop:cone_first_derivative}, we have that
    \begin{equation*}
        \lambda''(\varphi)[\eta, \zeta] = \left. \frac{d}{ds} \left(- \int_D e^{N(\varphi + s\zeta)} \eta \left(\frac{\partial u_{\varphi + s\zeta}}{\partial \nu}(e^{\varphi + s\zeta} q) \right)^2 \ d\sigma \right) \right|_{s = 0}.
    \end{equation*}
    By the Leibniz rule for differentiation under the integral, we obtain
    \begin{align}
        \lambda''(\varphi)[\eta, \zeta]
        & = - \int_D e^{N \varphi} N \eta \zeta \left(\frac{\partial \uphi}{\partial \nu} (e^\varphi q) \right)^2 \ d\sigma \nonumber \\
        & \quad - 2 \int_D e^{N \varphi} \eta \frac{\partial \uphi}{\partial \nu}(e^\varphi q) \left. \frac{d}{ds}\left(\frac{\partial u_{\varphi + s\zeta}}{\partial \nu} (e^{\varphi + s\zeta} q) \right) \right|_{s = 0} \ d\sigma. \label{eq:cone_second_variation_a}
    \end{align}

    Now, as shown in the proof of \cite[Lemma 3.2]{IacopettiPacellaWeth2022}, we have that
    \begin{equation*}
        \frac{d}{ds} \left(\nabla u_{\varphi + s\zeta} \right) = \nabla \left(\frac{d}{ds} u_{\varphi + s\zeta} \right),
    \end{equation*}
    and therefore
    \begin{align}
        & \left.\frac{d}{ds} \left(\frac{\partial u_{\varphi + s\zeta}}{\partial \nu}(e^{\varphi + s\zeta}q) \right)\right|_{s = 0} \nonumber \\
        & = \left(\nabla \wtu_\zeta(e^\varphi q) + D^2 \uphi (e^\varphi q) e^\varphi \zeta q \right) \cdot \frac{q - \nabla_\sphere \varphi}{\sqrt{1 + |\nabla_\sphere \varphi|^2}} \nonumber \\
        & \quad + \nabla \uphi(e^\varphi q) \cdot \left(- \frac{\nabla_\sphere \zeta}{\sqrt{1 + |\nabla_\sphere \varphi|^2}} - \frac{(q - \nabla_\sphere \varphi)(\nabla_\sphere\varphi \nabla_\sphere\zeta)}{(1 + |\nabla_\sphere \varphi|^2)^{\frac{3}{2}}} \right) \nonumber \\
        & = \left(\nabla \wtu_\zeta(e^\varphi q) + D^2 \uphi(e^\varphi q) e^\varphi \zeta q \right) \cdot \nu \nonumber \\
        & \quad + \nabla \uphi(e^\varphi q) \cdot \left(- \frac{\nabla_\sphere \zeta}{\sqrt{1 + |\nabla_\sphere \varphi|^2}} - \frac{\nabla_\sphere\varphi \nabla_\sphere\zeta}{(1 + |\nabla_\sphere \varphi|^2)} \nu \right) \nonumber \\
        & = \frac{\partial \wtu_\zeta}{\partial \nu}(e^\varphi q) + \zeta [D^2 \uphi (e^\varphi q) e^\varphi q] \cdot \nu - \frac{\nabla \uphi (e^\varphi q) \nabla_\sphere \zeta}{\sqrt{1 + |\nabla_\sphere \varphi|^2}} - \frac{\partial \uphi}{\partial \nu} (e^\varphi q) \frac{\nabla_\sphere\varphi \nabla_\sphere\zeta}{(1 + |\nabla_\sphere \varphi|^2)}. \label{eq:cone_second_variation_b}
    \end{align}

    Combining \eqref{eq:cone_second_variation_a} and \eqref{eq:cone_second_variation_b} we obtain \eqref{eq:cone_lambda_primeprime}.
\end{proof}

\subsection{Volume constrained critical points for $\lambda_1$}

Let $\varphi \in \ctd$ and consider the corresponding domain $\Op$. Its volume (Lebesgue measure) is given by
\begin{equation*}
    \vol(\Op) = |\Op| = \frac{1}{N} \int_D e^{N \varphi} \ d\sigma.
\end{equation*}
We can then define the volume functional
\begin{equation*}
    \vol: \varphi \in \ctd \mapsto \vol(\Op) \in \R.
\end{equation*}
It can easily be shown that $\vol$ is of class $C^2$ in $\ctd$ and for any $\eta, \zeta \in \ctd$ it holds
\begin{align}
    & \vol'(\varphi)[\eta] = \int_D e^{N\varphi} \eta \ d\sigma, \label{eq:cone_vol_prime} \\
    & \vol''(\varphi)[\eta, \zeta] = N \int_D e^{N \varphi} \eta \zeta \ d \sigma. \label{eq:cone_vol_primeprime}
\end{align}

For $m > 0$, we define
\begin{equation*}
    \M \coloneqq \{\varphi \in \ctd \ : \ \vol(\varphi) = m \},
\end{equation*}
which is a smooth manifold in $\ctd$, whose tangent space at a given point $\varphi \in \M$ is given by
\begin{equation*}
    T_\varphi\M = \left\{\eta \in \ctd \ : \ \int_D e^{N \varphi} \eta \ d\sigma = 0 \right\}.
\end{equation*}

For our analysis, it is natural to restrict the functional $\lambda_1(\varphi)$ to $\M$. 

By the theorem of Lagrange multipliers, if $\varphi \in \M$ is a critical point of the constrained functional $\lambda_1|_\M$, then there exists a Lagrange multiplier $\Lambda \in \R$ such that
\begin{equation}
    \label{eq:cone_Lagrange_multiplier_a}
    \lambda_1'(\varphi) = \Lambda \vol'(\varphi).
\end{equation}
As a consequence, we obtain the following result:
\begin{lemma}
    \label{lem:cone_overdet_critical_point}
    Let $\varphi \in \M$ be a critical point for $\lambda_1|_\M$. Then the Lagrange multiplier $\Lambda$ is negative and 
    \begin{equation*}
        \frac{\partial \uphi}{\partial \nu} = - \sqrt{- \Lambda} \quad \text{ on } \Gzp. 
    \end{equation*}
\end{lemma}
\begin{proof}
    Starting from \eqref{eq:cone_Lagrange_multiplier_a} and applying Proposition \ref{prop:cone_first_derivative}, we have that
    \begin{equation*}
        \int_D e^{N \varphi} \eta \left(\left(\frac{\partial \uphi}{\partial \nu}(e^\varphi q) \right)^2 + \Lambda \right) d\sigma = 0
    \end{equation*}
    for any $\eta \in \ctd$. Since $e^{N \varphi} > 0$ and $\eta$ is arbitrary, it immediately follows that $\Lambda \leq 0$ and
    \begin{equation*}
        \left(\frac{\partial \uphi}{\partial \nu} \right)^2 = - \Lambda \quad \text{ on } \Gzp.
    \end{equation*}
    By standard elliptic regularity for eigenfunctions we obtain that $\uphi$ is smooth (at least of class $C^2$) up to $\Gzp$, and since $\uphi > 0$ in $\Op$, by Hopf's lemma we have that $\mfrac{\partial \uphi}{\partial \nu} < 0$ on $\Gzp$. The proof is complete.
\end{proof}

\begin{remark}
    It follows at once from Lemma \ref{lem:cone_overdet_critical_point} that critical points of $\lambda_1|_\M$ are domains where the overdetermined problem \eqref{eq:cone_overdet_pde} admits a solution. Moreover, it recovers (with an easier proof\footnote{See, e.g., \cite[Proposition 2.6]{AfonsoIacopettiPacella2024Energypublished}.}) in the class of star-shaped domains, a standard result about (smooth) critical shapes for the first (more generally, simple) eigenvalues of the mixed Dirichlet-Neumann Laplacian under a volume constraint. 
\end{remark}

It can also be shown that the Lagrange multiplier allows us to compute the second derivative of $\lambda_1|_\M$ for variations on the tangent space. Precisely, we have the following result:
\begin{lemma}
    \label{lem:cone_Lagrange_multiplier_second_derivative}
    Let $\varphi \in \M$ be a critical point for $\lambda_1|_\M$ and let $\eta, \zeta \in T_\varphi\M$. Then
    \begin{equation*}
        \lambda_1|_\M''(\varphi)[\eta, \zeta] = \lambda_1''(\varphi)[\eta, \zeta] - \Lambda \vol''(\varphi)[\eta, \zeta].
    \end{equation*}
\end{lemma}
\begin{proof}
    The proof is the same as in \cite[Lemma 4.3]{IacopettiPacellaWeth2022}.
\end{proof}

The sign of $\lambda_1|_\M''(\varphi)$ at a critical point $\varphi \in \M$ determines the stability of $\Omega_\varphi$ as a critical domain for $\Omega \mapsto \lambda_1(\Omega)$ in the class of polar graphs.

\subsection{Stability of the spherical sector}

We will now focus on the spherical sector
\begin{equation*}
    \Omega_D \coloneqq \cone \cap B_1(O),
\end{equation*}
which corresponds, in the notation of \eqref{eq:cone_def_Omega_varphi}, to $\Omega_0$, that is, $\varphi \equiv 0$ on $D = \Gamma_{0, 0}$. We set the volume constraint
\begin{equation*}
    m = |\Omega_D|.
\end{equation*}
In this case, in view of \eqref{eq:cone_vol_prime}, the tangent space to $\M$ at $\varphi \equiv 0$ is
\begin{equation*}
    T_0\M = \{\eta \in \ctd \ : \ \int_D \eta \ d\sigma = 0\}.
\end{equation*}
For simplicity of notation, we also set $\Gamma_D = \pc \cap B_1(O)$. 

In this case, the first eigenfunction $u_D \coloneqq u_{\Omega_D}$ is simply the restriction of an appropriate normalization of $u_{B_1(O)}$, the first Dirichlet eigenfunction in $B_1(O)$ (which is well-known to be radial), to the cone $\cone$.

Since $u_D$ is radial, it is a solution of \eqref{eq:cone_overdet_pde} in $\Omega_D$, which therefore is a critical shape for $\lambda_1|_\M$ (with the volume constraint $m = |\Omega_D|$). Our aim is to understand conditions under which $\Omega_D$ is not a minimizer, in which case we may search for more exotic, non-radial solutions for the overdetermined problem \eqref{eq:cone_overdet_pde} by minimizing the functional $\Omega \mapsto \lambda_1(\Omega)$.

Hereafter, we write $r = |x|$ and $v(r) = v(|x|)$ for any radial function $v$ in $\Omega_D$, and we denote with a prime the derivative with respect to the radial variable. 

Since $u_D$ is radial, $\nu = q$ on $D$, and in view of Lemma \ref{lem:cone_overdet_critical_point}, we have that
\begin{equation*}
    \left. \frac{\partial u_D}{\partial \nu}\right|_D = u_D'(1) = - \sqrt{- \Lambda_D},
\end{equation*}
where $\Lambda_D$ is the Lagrange multiplier in the case $\varphi \equiv 0$. Moreover, also
\begin{equation*}
    \frac{\partial^2 u_D}{\partial \nu^2} = (D^2u_D \nu \cdot \nu)|_D = u_D''(1)
\end{equation*}
is constant on $D$.

Recall that for $\eta \in T_0\M$, since $\Omega_D$ is a critical shape for $\lambda_1|_\M$, then $\lambda_1|_\M'(0) = 0$, and the corresponding derivative eigenfunction is the solution $\wtu_\eta \in H^1(\Omega_D)$ to
\begin{equation}
    \label{eq:cone_pde_u_eta_prime_spherical_sector}
    \left\{
    \begin{array}{rcll}
        - \Delta \wtu_\eta & = & \lambda_1(\OD) \wtu_\eta & \quad \text{ in } \OD \\
        \wtu_\eta & = & -u_D'(1) \eta & \quad \text{ on } D \\
        \displaystyle \mfrac{\partial \wtu_\eta}{\partial \nu} & = & 0 & \quad \text{ on } \Gamma_D \setminus \{O\} \\
        \displaystyle \int_\OD \wtu_\eta u_D \ dx & = & 0 & 
    \end{array}
    \right.
    .
\end{equation}

Our next result shows that the quadratic form for $\lambda_1|_\M''(0)$ has a simple expression.

\begin{lemma}
    \label{lem:cone_simplified_second_derivative}
    For any $\eta \in T_0\M$ it holds
    \begin{equation}
        \label{eq:cone_simplified_second_derivative}
        \lambda_1|_\M''(0)[\eta, \eta] = - 2 u_D'(1) \left(\int_D \eta \frac{\partial \wtu_\eta}{\partial \nu} \ d\sigma + u_D''(1) \int_D \eta^2 \ d\sigma \right).
    \end{equation}
\end{lemma}

\begin{proof}
    We start from Proposition \ref{prop:cone_second_derivative} and Lemma \ref{lem:cone_Lagrange_multiplier_second_derivative}. In view of \eqref{eq:cone_vol_primeprime}, we have:
    \begin{align}
        \lambda_1|_\M''(0)[\eta, \eta] 
        & = - N \int_D \eta^2 (u_D'(1))^2 \ d\sigma - 2 \int_D u_D'(1) \eta \frac{\partial \wtu_\eta}{\partial \nu} \ d\sigma \nonumber \\
        & \quad - 2 \int_D u_D'(1) u_D''(1) \eta^2 \ d\sigma - \Lambda_D N \int_D \eta^2 \ d\sigma. \label{eq:cone_proof_simplified_second_derivative_a}
    \end{align}
   Since $\Lambda_D = - (u_D'(1))^2$, the first and last terms cancel, yielding \eqref{eq:cone_simplified_second_derivative}.
\end{proof}

To understand the stability of $\Omega_D$ as a stationary domain for the functional $\varphi \in \M \mapsto \lambda_1|_\M(\varphi) \in \R$, it is a key point to understand the behaviour of the solution $\wtu_\eta$ to \eqref{eq:cone_pde_u_eta_prime_spherical_sector}, for any $\eta \in T_0\M$. As already pointed out in \cite{IacopettiPacellaWeth2022, AfonsoIacopettiPacella2024Energypublished}, it is enough for this purpose to consider the Neumann eigenfunctions $\psi \in \ctd$ of the Laplace-Beltrami operator $- \Delta_\sphere$ on $D$, namely, the solutions $(\psi, \mu)$ of the problem
\begin{equation}
    \label{eq:cone_Neumann_eigenvalue_problem_D}
    \left\{
    \begin{array}{rcll}
        - \Delta_\sphere \psi & = & \mu \psi & \quad \text{ on } D \\
        \displaystyle \mfrac{\partial \psi}{\partial \nu_{\partial D}} & = & 0 & \quad \text{ on } \partial D
    \end{array}
    \right.
    ,
\end{equation}
where $\nu_{\partial D}$ denotes the outer unit normal vector to $\partial D$ on $\sphere$. This problem admits a sequence of eigenvalues, which we denote (counting with multiplicity) by
\begin{equation*}
    0 = \mu_0 < \mu_1 \leq \mu_2 \leq \ldots \nearrow + \infty.
\end{equation*}
To each eigenvalue $\mu_j$, there corresponds an $L^2(D)$-normalized eigenfunction $\psi_j$, for $j \in \N \cup \{0\}$, and the family $\{\psi_j\}_{j \in \N \cup \{0\}}$ is an orthonormal basis for $L^2(D)$. In particular, $\psi_0$ is constant and, most importantly,
\begin{equation*}
    \int_D \psi_j \ d\sigma = 0 \quad \forall j \geq 1.
\end{equation*}
We refer to \cite{Chavel1984} for more details on this problem.

We will now show that for any $\psi_j$, the corresponding solution to \eqref{eq:cone_pde_u_eta_prime_spherical_sector}, which we denote by
\begin{equation*}
    \wtu_j \coloneqq \wtu_{\psi_j},
\end{equation*}
admits a suitable separation of variables in the polar coordinates $(r, q) \in [0, 1] \times \overline D$. The first step in this direction is the following result:

\begin{prop}
    \label{prop:cone_def_h_j}
    Let $j \geq 1$ and let $\wtu_j$ be the solution of \eqref{eq:cone_pde_u_eta_prime_spherical_sector} with $\eta = \psi_j$. Define
    \begin{equation*}
        h_j(r) = \int_D \wtu_j(r, q) \psi_j(q) \ d\sigma, \quad r \in (0, 1).
    \end{equation*}
    Then $h_j$ satisfies
    \begin{equation}
        \label{eq:cone_pde_hj}
        \left\{
        \begin{array}{rcll}
            - h_j'' - \dfrac{N - 1}{r} h_j' - \lambda_1(\OD) h_j & = & - \dfrac{\mu_j}{r^2} h_j & \quad \text{ in } (0, 1) \\
            h_j(1) & = & - u_D'(1) &
        \end{array}
        \right.
        .
    \end{equation}
\end{prop}
\begin{proof}
    The proof is the same as in \cite[Theorem 3.10]{AfonsoIacopettiPacella2024Energypublished}. We reproduce it here for completeness and the convenience of the reader.

    Since $\wtu_j = - u_D'(1) \psi_j$ on $D$, it immediately follows that
    \begin{equation*}
        h_j(1) = - u_D'(1).
    \end{equation*}
    Next, for any $r \in (0, 1]$, the Leibniz rule allows us to take the derivatives with respect to $r$ under the integral sign, so that
    \begin{align}
        - h_j''(r) - \frac{N - 1}{r} h_j'(r)
        & = \int_D \left(- \frac{\partial^2 \wtu_j}{\partial r^2} (r, q) - \frac{N - 1}{r}\frac{\partial \wtu_j}{\partial r}(r, q) \right) \psi_j(q) \ d\sigma \nonumber \\
        & = \int_D \left(- \Delta \wtu_j(r, q) + \frac{1}{r^2} \Delta_\sphere \wtu_j(r, q) \right) \psi_j(q) \ d\sigma \nonumber \\
        & = \int_D \lambda_1(\Omega_D) \wtu_j(r, q) \psi_j(q) \ d\sigma + \frac{1}{r^2} \int_D \Delta_\sphere \wtu_j(r, q) \psi_j(q) \ d\sigma \nonumber \\
        & = \lambda_1(\OD) h_j(r) + \frac{1}{r^2} \int_D \wtu_j (\Delta_\sphere \psi_j) \ d\sigma \nonumber \\
        & = \lambda_1(\OD)h_j(r) - \frac{\mu_j}{r^2} h_j. \nonumber
    \end{align}
    The proof is complete.
\end{proof}

\begin{cor}
    Let $j \geq 1$ and let $\wtu_j$ be the solution of \eqref{eq:cone_pde_u_eta_prime_spherical_sector} with $\eta = \psi_j$. Then
    \begin{equation*}
        \wtu_j(r, q) = h_j(r) \psi_j(q), 
    \end{equation*}
    where $h_j$ is the solution of \eqref{eq:cone_pde_hj}.
\end{cor}

\begin{proof}
    Follows immediately from the definitions (see also \cite[Remark 3.11]{AfonsoIacopettiPacella2024Energypublished}).
\end{proof}

We have the following result regarding the regularity of $h_j$:

\begin{lemma}
    \label{lem:cone_reg_h_j}
    For any $j \geq 1$, it holds that
    \begin{equation}
        \label{eq:cone_estimate_hj_a}
        \int_0^1 r^{N - 3} h_j^2 \ dr < + \infty
    \end{equation}
    and
    \begin{equation}
        \label{eq:cone_estimate_hj_b}
        \int_0^1 r^{N - 1} (h_j')^2 \ dr < + \infty.
    \end{equation}
    Moreover, $h_j \in L^\infty(0, 1)$ and $h_j(0) = 0$.
\end{lemma}

\begin{proof}
    We follow \cite[Proposition 3.12]{AfonsoIacopettiPacella2024Energypublished}.

    Since $\wtu_j \in H^1(\OD)$ and $\psi_j$ is an $L^2$-normalized solution to \eqref{eq:cone_Neumann_eigenvalue_problem_D}, then passing to polar coordinates we obtain
    \begin{align}
        + \infty
        & > \int_\OD |\nabla \wtu_j|^2 \ dx \nonumber \\
        & = \left(\int_0^1 r^{N - 1}(h_j')^2 \ dr\right)\left(\int_D \psi_j^2 \ d\sigma\right) + \left(\int_0^1 r^{N - 3} h_j^2 \ dr \right) \left(\int_D |\nabla_\sphere \psi_j|^2 \ d\sigma \right) \nonumber \\
        & = \int_0^1 r^{N - 1}(h_j')^2 \ dr + \mu_j \int_0^1 r^{N - 3} h_j^2 \ dr, \nonumber
    \end{align}
    which immediately implies \eqref{eq:cone_estimate_hj_a} and \eqref{eq:cone_estimate_hj_b}.

    With these estimates at hand, we can proceed as in \cite[Lemma A.9]{DancerGladialiGrossi2017} to conclude.
\end{proof}

\begin{lemma}
    There exists a constant $C > 0$ such that, for any $\eta \in T_0\M$, the solution $\wtu_\eta$ to \eqref{eq:cone_pde_u_eta_prime_spherical_sector} satisfies
    \begin{equation}
        \label{eq:cone_estimate_u_eta}
        \|\wtu_\eta\|_{H^1(\OD)} \leq C \|\eta\|_{H^{1/2}(D)}.
    \end{equation}
\end{lemma}
\begin{proof}
    By the theory of traces (see, e.g., \cite{Wloka1987, Leoni2017}), we can choose a lift $\phi \in H^1(\OD)$ such that $\phi = - u_D'(1)\eta$ on $D$ (in the sense of traces) and
    \begin{equation}
        \label{eq:cone_proof_estimate_a}
        \|\phi\|_{H^1(\OD)} \leq C_1 \|\eta\|_{H^{1/2}(D)}.
    \end{equation}
    By composing with the orthogonal projection operator onto the space $\langle u_D\rangle^\perp$ if necessary, we can further assume that
    \begin{equation*}
        \int_\OD \phi u_D \ dx = 0.
    \end{equation*}

    With these ingredients at hand, finding a weak solution $\wtu_\eta \in H^1(\OD)$ to \eqref{eq:cone_pde_u_eta_prime_spherical_sector} is equivalent to finding a weak solution $\wtv \in (H_0^1(\OD \cup \Gamma_D) \cap \langle u_D \rangle^\perp)$ to
    \begin{equation}
        \label{eq:cone_variational_pde_v}
        \int_\OD \nabla \wtv \nabla v \ dx - \int_\OD \lambda_1(\OD) \wtv v \ dx = - \Phi(v) \quad \forall v \in H_0^1(\OD \cup \Gamma_D),
    \end{equation}
    where $\Phi \in H^{-1}(\OD \cup \Gamma_D)$ is the continuous linear functional
    \begin{equation*}
        \Phi(v) = \int_\OD \nabla \phi \nabla v \ dx - \lambda_1(\OD) \int_\OD \phi v \ dx.
    \end{equation*}
    To this end, we may use the Lax--Milgram theorem. Indeed, solutions are the critical points in the subspace $\langle u_D \rangle^\perp$ of the functional
    \begin{equation*}
        J(v) = \frac{1}{2} \int_\OD |\nabla v|^2 - \frac{1}{2}\lambda_1(\OD) \int_\OD v^2 \ dx + \Phi(v).
    \end{equation*}
    Note that $J$ is coercive in $\langle u_D \rangle^\perp$, since for every $v$ in this subspace it holds that
    \begin{equation*}
        \int_\OD (|\nabla v|^2 - \lambda_1(\OD) v^2) \ dx \geq \left(1 - \frac{\lambda_1(\OD)}{\lambda_2(\OD)} \right) \int_{\OD} |\nabla v|^2 \ dx.
    \end{equation*}
    Hence, by Lax-Milgram theorem, for each $\eta$ there exists a unique solution $\wtv$ to \eqref{eq:cone_variational_pde_v}, which moreover depends continuously on the datum $\phi$, that is, there exists some constant $C_2 > 0$ independent of $\phi$ such that
    \begin{equation}
        \label{eq:cone_proof_estimate_b}
        \|\wtv\|_{H_0^1(\OD \cup \Gamma_D)} \leq C_2 \|\phi\|_{H^1(\OD)},
    \end{equation}
    see \cite[Section 3.2.1]{Kesavan1989}. Then, writing $\wtu_\eta = \wtv + \phi$ and combining \eqref{eq:cone_proof_estimate_a} with \eqref{eq:cone_proof_estimate_b}, we readily obtain \eqref{eq:cone_estimate_u_eta}.
\end{proof}

The next result shows that we can decompose a solution $\wtu_\eta$ for \eqref{eq:cone_pde_u_eta_prime_spherical_sector} in terms of the functions $\wtu_j$.

\begin{lemma}
    \label{lem:cone_decomposition_u_eta}    
    Let $\eta \in T_0\M$ and let $\wtu_\eta$ be the solution of \eqref{eq:cone_pde_u_eta_prime_spherical_sector}. Then
    \begin{equation*}
        \wtu_\eta = \sum_{j = 1}^\infty (\eta, \psi_j) \wtu_j
    \end{equation*}
    where $(\cdot, \cdot)$ denotes the inner product in $L^2(D)$.
\end{lemma}

\begin{proof}
    We begin by recalling that since $\{\psi_j\}_{j \in \N \cup \{0\}}$ is an orthonormal basis for $L^2(D)$, then
    \begin{equation*}
        \eta = \sum_{j = 1}^\infty (\eta, \psi_j) \psi_j.
    \end{equation*}
    Without loss of generality, we assume that $\int_D \eta^2 \ d\sigma = 1$, so
    \begin{equation*}
        \sum_{j = 1}^\infty (\eta, \psi_j)^2 = 1.
    \end{equation*}
    For each $k \in \N$, let
    \begin{equation*}
        u_k = \sum_{j = 1}^k (\eta, \psi_j) \wtu_j.
    \end{equation*}
    Since each $u_k$ is a finite sum, it is clear that $u_k \in H^1(\OD)$ and is a solution of \eqref{eq:cone_pde_u_eta_prime_spherical_sector} with boundary condition
    \begin{equation*}
        u_k = - u_D'(1)\left(\sum_{j = 1}^k(\eta, \psi_j)\psi_j\right) \quad \text{ on } D.
    \end{equation*}
    Then, by making use of the regularity of $\eta$ and \eqref{eq:cone_estimate_u_eta}, we obtain
    \begin{align}
        0
        & \leq \limsup_{k \to \infty} \|\wtu_\eta - u_k\|_{H^1(\OD)} \nonumber \\
        & \leq \limsup_{k \to \infty} C \left\|\sum_{j = k + 1}^\infty (\eta, \psi_j) \psi_j\right\|_{H^{1/2}(D)} \nonumber \\
        & = 0. \nonumber
    \end{align}
    Now, by definition, 
    \begin{equation*}
        u_k \to \sum_{j = 1}^\infty (\eta, \psi_j) \wtu_j
    \end{equation*}
    strongly in $H^1(\OD)$ as $k \to \infty$ (see \cite[Lemma 5.1]{Brezis2010}), which completes the proof.
\end{proof}

Then, we have the following positivity result:
\begin{lemma}
    \label{lem:cone_positivity_h_j}
    For all $j \in \N$, it holds that
    \begin{equation}
        \label{eq:cone_positivity_h_j}
        h_j > 0 \quad \text{ in } (0, 1).
    \end{equation}
\end{lemma}

\begin{proof}
    Suppose that $h_j(r_0) = 0$ for some $r_0 \in (0, 1)$. Then $h_j$ satisfies
    \begin{equation*}
        \left\{
        \begin{array}{rcll}
            - (r^{N - 1} h_j')' - \lambda_1(\OD) r^{N - 1} h_j & = & - \mu_j r^{N - 3} h_j & \quad \text{ in } (0, r_0) \\
            h_j(0) & = & 0 & \\
            h_j(r_0) & = & 0 &
        \end{array}
        \right.
    \end{equation*}
    Multiplying by $h_j$ and integrating in $(0, r_0)$, we obtain
    \begin{align}
        \int_0^{r_0} r^{N - 1} (h_j')^2 
        & = \lambda_1(\OD) \int_0^{r_0} r^{N - 1} h_j^2 \ dr - \mu_j \int_0^{r_0} r^{N - 3} (h_j)^2 \ dr \nonumber \\
        & < \lambda_1(\OD) \int_0^{r_0} r^{N - 1} h_j^2 \ dr, 
    \end{align}
    contradicting $\lambda_1((0, r_0)) > \lambda_1((0, 1))$.
\end{proof}

In the next definition, we set in a precise way what we mean by a stable/unstable critical shape.

\begin{definition}
    Let $\varphi \in \M$. The domain $\Omega_\varphi$ is said to be a stable critical domain if $\varphi$ is a critical point for $\varphi \in \M \mapsto \lambda_1|_\M(\varphi)$ and the second derivative is a positive definite quadratic form, that is,
    \begin{equation}
        \label{eq:cone_critical_domain_def}
        \lambda_1|_\M'(\varphi)[\eta] = 0 \quad \forall \eta \in T_\varphi\M
    \end{equation}
    and
    \begin{equation*}
        \lambda_1|_\M''(\varphi)[\eta, \eta] > 0 \quad \forall \eta \in T_\varphi\M, \ \eta \not\equiv 0.
    \end{equation*}
    On the contrary, $\Omega_\varphi$ is said to be an unstable critical domain if \eqref{eq:cone_critical_domain_def} holds and there exists some $\eta \in T_\varphi\M$ such that
    \begin{equation*}
        \lambda_1|_\M''(\varphi)[\eta, \eta] < 0.
    \end{equation*}
\end{definition}

We now prove the main result of this section. It is interesting to note that the threshold for stability is precisely the same as that found by \cite{IacopettiPacellaWeth2022} for the torsion problem and \cite{AfonsoIacopettiPacella2024Energypublished} for autonomous nonlinear problems.

\begin{theorem}
    \label{thm:cone_stability}
    Let $\cone$ be the cone spanned by the smooth domain $D \subset \sphere$, with $N \geq 3$, and let $\mu_1$ be the first non-zero Neumann eigenvalue of the Laplace--Beltrami operator on $D$.
    \begin{enumerate}[label=(\roman*)]
        \item If $\mu_1 < N - 1$, then $\OD$ is an unstable critical domain.

        \item If $\mu_1 > N - 1$, then $\Omega_D$ is a stable critical domain.
    \end{enumerate}
\end{theorem}

\begin{proof}
    We argue as in \cite[Theorem 1.1]{AfonsoIacopettiPacella2024Energypublished}.

    We first prove \textit{(i)}. Let $\wtu_1(r, q) = h_1(r) \psi_1(q)$ in $\OD$ be the solution of \eqref{eq:cone_pde_u_eta_prime_spherical_sector}. By Lemma \ref{lem:cone_simplified_second_derivative}, we have
    \begin{equation}
        \label{eq:cone_proof_main_lambda_primeprime}
        \lambda_1|_\M''(0)[\psi_1, \psi_1] = - 2 u_D'(1)(h_1'(1) + u_D''(1)).
    \end{equation}

    Writing the equation for $h_1$ in Sturm--Liouville form, we obtain
    \begin{equation}
        \label{eq:cone_proof_main_equation_h_1_SL}
        - (r^{N - 1} h_1')' - r^{N - 1} \lambda_1(\OD) h_1 = - r^{N - 3} \mu_1 h_1.
    \end{equation}
    On the other hand, writing $- \Delta u_D = \lambda_1(\OD) u_D$ in polar coordinates and taking the derivative with respect to $r$, we obtain
    \begin{equation*}
        - (u_D')'' - \frac{N - 1}{r} (u_D')' - \lambda_1(\OD) u_D' = - \frac{N - 1}{r^2} u_D', 
    \end{equation*}
    which in Sturm--Liouville form is
    \begin{equation}
        \label{eq:cone_proof_main_equation_u_D_prime_SL}
        - (r^{N - 1} u_D'')' - r^{N - 1} \lambda_1(\OD) u_D' = - r^{N - 3} (N - 1) u_D'.
    \end{equation}

    Let $\bar r \in (0, 1)$. Multiplying \eqref{eq:cone_proof_main_equation_h_1_SL} by $u_D'$ and integrating by parts in $(\bar r, 1)$ we obtain
    \begin{align}
        \int_{\bar r}^1 r^{N - 1} h_1' u_D'' \ dr - (r^{N - 1} h_1' u_D')|_{\bar r}^1 - \lambda_1(\OD) \int_{\bar r}^1 r^{N - 1}
        u_D' h_1 \ dr = - \mu_1 \int_{\bar r}^1 r^{N - 3} h_1 u_D' \ dr. \label{eq:cone_proof_main_3.43}
    \end{align}
    Similarly, multiplying \eqref{eq:cone_proof_main_equation_u_D_prime_SL} by $h_1$ and integrating by parts we deduce that
    \begin{align}
        \int_{\bar r}^1 r^{N - 1} h_1' u_D'' \ dr - (r^{N - 1} h_1 u_D'')|_{\bar r}^1 - \lambda_1(\OD) \int_{\bar r}^1 r^{N - 1} u_D' h_1 \ dr = - (N - 1) \int_{\bar r}^1 r^{N - 3} u_D' h_1 \ dr. \label{eq:cone_proof_main_3.44}
    \end{align}

    Combining Lemma \ref{lem:cone_reg_h_j} with Hölder's inequality, and since $u_D'$ is bounded in $[0, 1]$, the right-hand sides of \eqref{eq:cone_proof_main_3.43} and \eqref{eq:cone_proof_main_3.44} remain finite in the limit as $\bar r \to 0$.

    Next, we claim that
    \begin{equation}
        \label{eq:cone_proof_main_claim}
        \lim_{\bar r \to 0^+} \bar r^{N - 1} h_1'(\bar r) u_D'(\bar r) = 0. 
    \end{equation}
    To prove the claim, we integrate \eqref{eq:cone_proof_main_equation_h_1_SL} to obtain
    \begin{align}
        \left|\int_{\bar r}^1 -(r^{N - 1} h_1')' \ dr \right|
        & = \left|\bar r^{N - 1} h_1'(\bar r) - h_1'(1) \right| \nonumber \\
        & \leq \lambda_1(\OD) \int_{\bar r}^1 r^{N - 1} h_1 \ dr + \int_{\bar r}^1 r^{N - 3} \mu_1 h_1 \ dr \nonumber \\
        & \leq C_1 \nonumber
    \end{align}
    for some $C_1 > 0$. As a consequence, there exists $C_2 > 0$ such that
    \begin{equation*}
        \limsup_{\bar r \to 0^+} \bar r^{N - 1} |h_1'(\bar r)| \leq C_2.
    \end{equation*}
    Then \eqref{eq:cone_proof_main_claim} follows from the fact that
    \begin{equation*}
        \lim_{\bar r \to 0^+} u_D'(\bar r) = u_D'(0) = 0,
    \end{equation*}
    because $u_D = u_{B_1(O)}|_\cone$ and $\nabla u_{B_1(O)}(O) = 0$. 

    Now, subtracting \eqref{eq:cone_proof_main_3.44} from \eqref{eq:cone_proof_main_3.43} and taking the limit as $\bar r \to 0^+$ it follows that
    \begin{equation}
        \label{eq:cone_proof_main_main_eq}
        - u_D'(1)(h_1'(1) + u_D''(1)) = (N - 1 - \mu_1) \int_0^1 r^{N - 3} h_1 u_D' \ dr < 0,
    \end{equation}
    because $\mu_1 < N - 1$, $h_1 > 0$ and $u_D' < 0$ in $(0, 1]$. Hence, in view of \eqref{eq:cone_proof_main_lambda_primeprime}, it holds that
    \begin{equation*}
        \lambda_1|_\M''(0)[\psi_1, \psi_1] < 0,
    \end{equation*}
    which concludes the proof of \textit{(i)}.

    To prove \textit{(ii)}, we recall that the family $\{\psi_j\}_{j \in \N \cup \{0\}}$ is an orthonormal basis for $L^2(D)$. Then any $\eta \in T_0\M$ can be written as
    \begin{equation}
        \label{eq:cone_proof_main_eta}
        \eta = \sum_{j = 1}^\infty (\eta, \psi_j) \psi_j,
    \end{equation}
    where $(\cdot, \cdot)$ denotes the inner product in $L^2(D)$. Without loss of generality, we may assume that $\int_D \eta^2 \ d\sigma = 1$. By Lemma \ref{lem:cone_decomposition_u_eta}, 
    \begin{equation*}
        \wtu_\eta = \sum_{j = 1}^\infty (\eta, \psi_j) \wtu_j.
    \end{equation*}

    Now we show that if $k > j$, then $h_k'(1) \geq h_j'(1)$, with strict inequality if $\mu_k > \mu_j$. To this end, we argue as in the proof of \textit{(i)}, namely, we write the equations for $h_k$ and $h_j$ in Sturm--Liouville form, multiply by $h_j$ and $ h_k$, respectively, integrate by parts, and subtract, to obtain
    \begin{equation*}
        - u_D'(1) (h_k'(1) - h_j'(1)) = (- \mu_j + \mu_k) \int_0^1 r^{N - 3} h_k h_j \ dr \geq 0.
    \end{equation*}
    This fact will be useful together with the following remark: for every $j \in \N$, it holds that
    \begin{equation*}
        \frac{\partial \wtu_j}{\partial \nu} (1, q) = h_j'(1) \psi_j(q).
    \end{equation*}

    We now compute
    \begin{align}
        \lambda_1|_\M''(0)[\eta, \eta]
        & = - 2u_D'(1) \left(\int_D \left(\sum_{j = 1}^\infty (\eta, \psi_j) \psi_j \right)\left(\sum_{k = 1}^\infty (\eta, \psi_k) h_k'(1) \psi_k \right) \ d\sigma + u_D''(1) \int_D \eta^2 \ d\sigma\right) \nonumber \\
        & = - 2u_D'(1) \left(\sum_{j = 1}^\infty (\eta, \psi_j)^2 h_j'(1) + u_D''(1)\right) \nonumber \\
        & \geq - 2u_D'(1) (h_1'(1) + u_D''(1)) \nonumber \\
        &  = 2 (N - 1 - \mu_1) \int_0^1 r^{N - 3} h_1 u_D' \ dr \nonumber \\
        & > 0, \nonumber
    \end{align}
    in view of \eqref{eq:cone_proof_main_main_eq}, and since $\mu_1 > N - 1$ (by assumption), $h_1 > 0$ and $u_D' < 0$ in $(0, 1]$. Since $\eta \in T_0\M$ is arbitrary, this shows that $\OD$ is a stable domain for the first eigenvalue. The proof is complete.
\end{proof}

\begin{remark}
    We refer to \cite[Section 5 and Appendix A]{IacopettiPacellaWeth2022} for examples of cones satisfying the condition $\mu_1 < N - 1$.
\end{remark}

\section{Existence of volume-constrained minimizers}
\label{sec:cone_existence}

Our aim in this section is to understand under which conditions the problem of minimizing $\lambda_1(\Omega)$ admits a solution (among domains of fixed measure). More precisely, for a fixed $m > 0$ we consider
\begin{equation}
    \label{eq:cone_min_A}
    \lambda(\cone, m) \coloneqq \inf\{\lambda_1(\Omega) \ : \ \Omega \subset \cone \text{ is quasi-open}, \ |\Omega| = m\}.
\end{equation}

\begin{remark}
    \label{rem:cone_scaling}
    Since the cone $\cone$ is invariant under scaling, we can invoke the well-known scaling properties of eigenvalues
    \begin{equation}
    \label{eq:cone_monotonicity}
    \lambda_1(t\Omega) = \frac{\lambda_1(\Omega)}{t^2}, \quad t\Omega = \{tx \in \R^N \ : \ x \in \Omega\}
    \end{equation}
    to conclude that either a minimizer exists for any fixed volume or there are no minimizers at all, for any volume bound chosen.
\end{remark}

\begin{remark}
    \label{rem:cone_sphere}
    Suppose that $D$ is a hemisphere, which, without loss of generality, can be thought of as the upper hemisphere
    \begin{equation*}
        \sphereplus = \sphere \cap \{x \in \R^N \ : \ x_N > 0\}.
    \end{equation*}
    In this case, the cone is the full half-space $\R^N_+ = \{x \in \R^N \ : \ x_N > 0\}$, and it is well known that symmetrization techniques are available, which allow us to conclude that the minimum in \eqref{eq:cone_min_A} is attained by any half-ball of measure $m$ with center in $\partial \R^N_+$.
\end{remark}

Next, we show that the minimum $\lambda(\cone, m)$ is indeed positive.

\begin{lemma}
    \label{lem:cone_min_not_zero}
    For any $m > 0$, it holds that 
    \begin{equation*}
        \lambda(\cone, m) > 0.
    \end{equation*}
\end{lemma}
\begin{proof}
    Since $\cone$ satisfies the cone condition, the Gagliardo--Nirenberg inequality holds: for every $u \in H_0^1(\Omega; \cone)$,
    \begin{equation*}
        \int_\cone |\nabla (u^2)| \ dx \geq C(N, D) \left( \int_\cone |u|^{\frac{2N}{N - 1}} \ dx \right)^{\frac{N - 1}{N}},
    \end{equation*}
    where $C(N, D)$ depends only on the dimension $N$, the geometry of the cone $\cone$ (hence on $D$). Next, Hölder's inequality on the right-hand side yields
    \begin{equation*}
        2 \|u\|_{L^2(\cone)} \|\nabla u\|_{L^2(\cone)} \geq C(N, D) \|u\|_2^2 \frac{1}{|\Omega|^{\frac{1}{N}}}.
    \end{equation*}
    Hence
    \begin{equation*}
        \lambda_1(\Omega) = \inf_{\substack{u \in H_0^1(\Omega \cup \Gamma_\Omega) \\ \|u\|_2 = 1}} \int_\cone |\nabla u|^2 \ dx \geq \frac{C(N, D)^2}{4m^{\frac{2}{N}}} > 0,
    \end{equation*}
    for every quasi-open $\Omega \subset \cone$ such that $|\Omega| = m$. Taking the infimum among these sets, we conclude the proof.
\end{proof}

We will now show that $\lambda(\Sigma_\sphereplus, m)$ is an upper bound for $\lambda(\cone, m)$, that is,
\begin{equation*}
    \lambda(\cone, m) \leq \lambda(\Sigma_\sphereplus, m) 
\end{equation*}
for every smooth $D \subset \sphere$. The next step will be to show that whenever the inequality is strict, the infimum in \eqref{eq:cone_min_A} is attained; see Theorem \ref{thm:cone_existence}.

\begin{prop}
    \label{prop:cone_bound_lambda_hemisphere}
    Let $D \subset \sphere$ be a smooth domain and let $m > 0$. Then
    \begin{equation}
        \label{eq:cone_bound_lambda_hemisphere}
        \lambda(\cone, m) \leq \lambda(\Sigma_\sphereplus, m).
    \end{equation}
\end{prop}

\begin{proof}
    We begin with some geometric considerations regarding the cone $\cone$.
    
    Fix $q \in \partial D \subset \partial \cone \setminus \{O\}$. Since $D$ is smooth, $\partial \cone$ is smooth near $q$, and therefore the tangent plane $T_q\partial \cone$ is well-defined. Let $(x', x_N)$ be the coordinates of points in $\R^N$ with respect to the coordinate system $\{v_1, \ldots, v_{N - 1}, - \nu(q)\}$, where $\{v_i\}_{i = 1, \ldots, N - 1}$ is an orthonormal basis for $T_q \pc$ and $- \nu(q)$ is the inner unit normal vector to $\pc$ at $q$.

    Observe that there exists an open neighborhood $V$ of $q$ in $\pc \setminus \{O\}$ and an open neighborhood $U$ of the origin $0$ in $T_q\pc$ such that $V - q$ is the graph of a smooth function $g: U \to \R$ (see \cite[Page 1034]{IacopettiPacellaWeth2022}). In particular, $\nabla_{x'} g(0) = 0$, where $\nabla_{x'}$ denotes the gradient with respect to the variables $x_1', \ldots, x_{N - 1}'$.

    Since $\cone$ is a cone, we have that for any $t > 0$ it holds
    \begin{equation*}
        T_{tq} \pc = T_q \pc, \quad \nu(tq) = \nu(q),
    \end{equation*}
    and moreover $tV - tq$ is the graph on $T_q \pc$ of the function $g_t: tU \to \R$ given by
    \begin{equation*}
        g_t(x') = t g\left(\frac{x'}{t} \right), \quad x' \in tU.
    \end{equation*}
    By the chain rule and the mean value theorem, for any $x' \in tU$ and $i = 1, \ldots, N - 1$ it holds
    \begin{align}
        \frac{\partial g_t}{\partial x_i'} (x')
        & = \frac{\partial g}{\partial x_i'} \left( \frac{x'}{t} \right) \nonumber \\
        & = \frac{\partial g}{\partial x_i'}(0) + \left(\nabla_{x'} \frac{\partial g}{\partial x_i'}\right) (\eta_i) \cdot \frac{x'}{t} \nonumber \\
        & = \left(\nabla_{x'} \frac{\partial g}{\partial x_i'}\right) (\eta_i) \cdot \frac{x'}{t}, \label{eq:cone_mean_value_theorem}
    \end{align}
    where $\eta$ belongs to the segment joining $0$ and  $\mfrac{x'}{t}$. It follows that for any fixed ball $B_R = B_R(0) \subset T_q \pc$ and all $t$ sufficiently large such that $B_R \subset tU$ it holds
    \begin{equation*}
        \max_{x' \in \overline{B_R}} |\nabla_{x'} g_t(x')| \leq \frac{1}{t} \max_{x' \in \overline U} \sqrt{\sum_{i = 1}^{N - 1} \left|\left(\nabla_{x'} \frac{\partial g}{\partial x_i'}\right)(x')\right|^2} R \leq \frac{C}{t},
    \end{equation*}
    where $C$ does not depend on $t$. In particular, it holds
    \begin{equation}
        \label{eq:cone_nabla_xprime_g_t_vanishes}
        \lim_{t \to + \infty} \max_{x' \in \overline{B_R}} |\nabla_{x'} g_t(x')| = 0.
    \end{equation}

    Next we consider the upper cylinder $C_R^+$ generated by $B_R$ and the epigraph $E_t^+$ of $g_t|_{B_R}$:
    \begin{align}
        & C_R^+ \coloneqq \{(x', x_N) \in \R^N \ : \ x' \in B_R, \ x_N > 0\}, \nonumber \\
        & E_t^+ \coloneqq \{(x', x_N) \in \R^N \ : \ x' \in B_R, \ x_N > g_t(x')\}. \nonumber 
    \end{align}
    We observe that the map $F_t: \overline{C_R^+} \to \overline{E_t^+}$ given by the formula
    \begin{equation*}
        F_t(x', x_N) = (x', x_N + g_t(x')), \quad (x', x_N) \in C_R^+
    \end{equation*}
    is a diffeomorphism, whose Jacobian matrix is given by
    \begin{equation*}
        J F_t(x', x_N) = 
        \begin{bmatrix}
            I_{N - 1} & 0_{N - 1}^T \\
            \nabla_{x'} g_t(x') & 1
        \end{bmatrix}
        ,
    \end{equation*}
    where $I_{N - 1}$ is the identity matrix of order $N - 1$ and $0_{N - 1}^T$ is the transpose of the null vector in $\R^{N - 1}$. In addition, it easily follows that 
    \begin{equation*}
        \det J F_t \equiv 1
    \end{equation*}
    and that the inverse map is given by
    \begin{equation*}
        F_t^{-1}(x', x_N) = (x', x_N - g_t(x')), \quad (x', x_N) \in E_t^+
    \end{equation*}

    We remark that $J F_t$ and $JF_t^{-1}$ do not depend on $x_N$. Moreover, both linear operators $J F_t$ and $JF_t^{-1}$ converge to the identity as $t \to + \infty$: in view of \eqref{eq:cone_nabla_xprime_g_t_vanishes}, it holds that
    \begin{equation}
        \label{eq:cone_JFt_to_id}
        \lim_{t \to + \infty} \|J F_t - I_N\| = \lim_{t \to + \infty} \|J F_t^{-1} - I_N\| = 0.
    \end{equation}

    Now we proceed to prove \eqref{eq:cone_bound_lambda_hemisphere}.

    Let $B^+ \subset \R^N_+$ be the $N$-dimensional half-ball such that $|B^+| = m$, thought of as contained in the upper half-space delimited by $T_q\pc \simeq \R^N_+$. Let $u_{B^+}$ be the first $L^2$-normalized eigenfunction of $B^+$, so that
    \begin{equation*}
        \lambda_1(B^+) = \int_{B^+} |\nabla u_{B^+}|^2 \ dx = \lambda(\Sigma_\sphereplus, m).
    \end{equation*}

    Let $B_R \subset T_q \pc$ be an $(N - 1)$-dimensional ball with a radius $R > 0$ large enough so that
    \begin{equation*}
        B^+ \subset B_R \times [0, + \infty).
    \end{equation*}
    For $t > 0$, we set 
    \begin{equation*}
        v_t = u_{B^+} \circ F_t^{-1}.
    \end{equation*}
    Henceforth we consider $t > 0$ large enough so that, by construction, $v_t \in H_0^1(F_t(B^+); \cone - tq)$.

    Since, as already noted, $\det JF_t \equiv 1$, it easily follows by the theorem of change of variables in the integral that 
    \begin{equation*}
        |F_t(B^+)| = |B^+| = m
    \end{equation*}
    and
    \begin{equation*}
        \int_{F_t(B^+)} v_t^2 \ dx = \int_{B^+} u_{B^+}^2 \ dx = 1.
    \end{equation*}
    Moreover, setting
    \begin{equation*}
        M_t \coloneqq \max_{(x', x_N) \in \overline{E_t^+}} \max_{|\zeta| \leq 1} |JF_t(x', x_N)^{-1} \zeta| = \max_{(x', x_N) \in \overline{E_t^+}} \|J F_t(x', x_N)^{-1}\|
    \end{equation*}
    and observing that 
    \begin{equation*}
        \nabla v_t(x', x_N) = [JF_t(x', x_N)^{-1}]^T \nabla u_{B^+}(F_t^{-1}(x', x_N)) \quad \text{ a.e. in } F_t(B^+),
    \end{equation*}
    we obtain
    \begin{equation*}
        \int_{F_t(B^+)} |\nabla v_t|^2 \ dx \leq M_t^2 \int_{B^+} |\nabla u_{B^+}|^2 \ dx.
    \end{equation*}
    Moreover, since \eqref{eq:cone_JFt_to_id} holds, then 
    \begin{equation*}
        \lim_{t \to + \infty} M_t = 1.
    \end{equation*}

    To conclude, we note that
    \begin{align}
        \lambda(\Sigma_\sphereplus, m) 
        & = \int_{B^+} |\nabla u_{B^+}|^2 \ dx \nonumber \\
        & \geq \frac{1}{M_t^2} \int_{F_t(B^+)} |\nabla v_t|^2 \ dx \nonumber \\
        & \geq \frac{1}{M_t^2} \lambda_1(F_t(B^+)) \nonumber \\
        & \geq \frac{1}{M_t^2} \lambda(\cone, m). \nonumber
    \end{align}
    Passing to the limit as $t \to \infty$ yields \eqref{eq:cone_bound_lambda_hemisphere} and concludes the proof.
\end{proof}

\begin{remark}
    Note that $\lambda_1(F_t(B^+))$ is an eigenvalue computed in the relative setting of the translated cone $\cone - tq$ rather than $\cone$, but the comparison with $\lambda(\cone, m)$ is valid because the problem is invariant if we translate both the ambient cone and the competitor set for $\lambda_1$.
\end{remark}

\begin{theorem}
    \label{thm:cone_existence}
    Let $D \subset \sphere$ be a smooth domain and let $m > 0$. If
    \begin{equation}
        \label{eq:cone_ineq}
        \lambda(\cone, m) < \lambda(\Sigma_\sphereplus, m), 
    \end{equation}
    then $\lambda(\cone, m)$ is attained.
\end{theorem}

\begin{proof}
    Let $\{\Omega_n\}_{n \in \mathbb N}$ be a minimizing sequence for \eqref{eq:cone_min_A}, that is, a sequence of quasi-open sets such that $|\Omega_n| = m$ for every $n \in \mathbb{N}$ and 
    \begin{equation*}
        \lambda_1(\Omega_n) = \int_{\Omega_n} |\nabla u_n|^2 \ dx \to \lambda(\cone, m) 
    \end{equation*}
    as $n \to \infty$, where $u_n \in H_0^1(\Omega_n; \cone)$ is the corresponding positive $L^2$-normalized eigenfunction (that is, $\|u_n\|_2 = 1$ for every $n \in \N$).

    By definition, the sequence $\{u_n\}_{n \in \N}$ is bounded in $H^1(\cone)$. Since $\|u_n\|_2 = 1$ for every $n \in \N$, we can apply a slight modification of the concentration-compactness lemma (\cite[Lemma III.1]{Lions1984}) to account for working in the Sobolev space $H^1(\cone)$ instead of $H^1(\R^N)$. More precisely, it holds that there exists a subsequence $\{u_{n_k}\}_{k \in \N}$ satisfying one of the following possibilities:
    \begin{enumerate}[label=(\roman*)]
        \item Compactness: there exists a sequence of points $y_k \in \cone$ such that 
        \begin{equation}
            \label{eq:cone_compactness}
            \forall \varepsilon > 0 \ \exists R > 0 \ \text{such that} \ \int_{B_R(y_k) \cap \cone} u_{n_k}^2 \ dx \geq 1 - \varepsilon \ \forall k \in \N;
        \end{equation}

        \item Vanishing:
        \begin{equation}
            \label{eq:cone_vanishing}
            \lim_{k \to \infty} \sup_{y \in \cone} \int_{B_R(y) \cap \cone} u_{n_k}^2 \ dx = 0 \quad \forall R > 0;
        \end{equation}

        \item Dichotomy: there exists $\alpha \in (0, 1)$ such that for all $\varepsilon > 0$ there exist $k_0 \geq 1$ and two sequences $\{u_{1, k}\}_{k \in \N} \subset H^1(\cone)$, $\{u_{2, k}\}_{k \in \N} \subset H^1(\cone)$, bounded in $H^1(\cone)$, such that for all $k \geq k_0$ it holds
        \begin{align}
            & \|u_{n_k} - u_{1, k} - u_{2, k}\|_{L^2(\cone)} \leq 4\varepsilon, \nonumber \\
            & \left|\int_\cone u_{1, k}^2 \ dx - \alpha\right| \leq \varepsilon, \quad \left|\int_\cone u_{2, k}^2 \ dx - (1 - \alpha) \right| \leq \varepsilon, \nonumber \\
            & \text{dist} (\supp u_{1, k}, \supp u_{2, k}) \to + \infty \quad \text{ as } k \to \infty, \nonumber \\
            & \liminf_{k \to \infty} \int_\cone (|\nabla u_{n_k}|^2 - |\nabla u_{1, k}|^2 - |\nabla u_{2, k}|^2) \ dx \geq 0. \nonumber
        \end{align}
    \end{enumerate}

    We now show that, under our assumptions, only the compactness situation may occur. For simplicity, we denote the subsequence $\{u_{n_k}\}_{k \in \N}$ by $\{u_k\}_{k \in \N}$, as well as $\Omega_k = \{u_k > 0\}$.

    To rule out the vanishing possibility, we proceed by contradiction. Hence we assume that \eqref{eq:cone_vanishing} holds. Then it follows by \cite[Lemma I.1]{Lions1984CCb} (see also \cite[Lemma 1.21]{Willem1996}), whose proof can be easily adapted to the case of the Sobolev space $H^1(\cone)$ (since $\cone$ satisfies the cone condition), that $u_k \to 0$ strongly in $L^p(\cone)$ as $k \to \infty$ for every $p \in (2, 2^*)$. Since $|\Omega_k| = m$, then Hölder's inequality would imply
    \begin{equation*}
        1 = \|u_k\|_2 \leq |\{u_k > 0\}|^{\frac{1}{2} - \frac{1}{p}} \|u_k\|_p \to 0 \quad \text{ as } k \to \infty,
    \end{equation*}
    absurd. Hence vanishing does not occur.

    With the aim of ruling out dichotomy, we observe that a diagonal argument together with the construction performed in the proof of \cite[Lemma III.1]{Lions1984} allows us to assume that the sequences $\{u_{1, k}\}_{k \in \N}$ and $\{u_{2, k}\}_{k \in \N}$ satisfy 
    \begin{align}
        & \supp u_{1, k} \cup \supp u_{2, k} \subseteq \supp u_{k} \quad \forall k \in \N \nonumber \\
        & \|u_{k} - u_{1, k} - u_{2, k}\|_{L^2(\cone)} \to 0 \quad \text{ as } k \to \infty, \nonumber \\
        & \int_\cone u_{1, k}^2 \ dx \to \alpha, \quad \int_\cone u_{2, k}^2 \ dx \to (1 - \alpha), \quad \text{ as } k \to \infty, \label{eq:cone_dichotomy} \\
        & \text{dist} (\supp u_{1, k}, \supp u_{2, k}) \to + \infty \quad \text{ as } k \to \infty, \nonumber \\
        & \liminf_{k \to \infty} \int_\cone (|\nabla u_{k}|^2 - |\nabla u_{1, k}|^2 - |\nabla u_{2, k}|^2) \ dx \geq 0. \nonumber
    \end{align}
    Moreover, the functions $u_{1, k}$ and $u_{2, k}$ can be assumed to be nonnegative for every $k \in \N$. Let us set $\Omega_{i, k} = \{u_{i, k} > 0\}$, for $i = 1, 2$, and define
    \begin{equation*}
        c_i \coloneqq \liminf_{k \to \infty} |\Omega_{i, k}|, \quad i = 1, 2.
    \end{equation*}
    We claim that $c_i > 0$ for $i = 1, 2$. Indeed, combining Hölder's and Sobolev inequalities, we obtain:
    \begin{align}
        \int_{\cone} u_{i, k}^2 \ dx 
        & \leq \left(\int_\cone u_{i, k}^{2^*} \right)^{\frac{2}{2^*}} |\Omega_{i, k}|^{\frac{2}{N}} \nonumber \\
        & \leq C(N, \cone) |\Omega_{i, k}|^{\frac{2}{N}} \int_\cone |\nabla u_{i, k}|^2 \ dx. \nonumber
    \end{align}
    Should we have $c_i = 0$, we would be in contradiction with \eqref{eq:cone_dichotomy}. Next, arguing as in \cite[Section 3.3]{Bucur2000}, we obtain that $u_{1, k} + u_{2, k}$ is a minimizing sequence for \eqref{eq:cone_min_A}: since $u_{i, k} \in H_0^1(\Omega_k; \cone)$, $i = 1, 2$, for every $k \in \N$, we have
    \begin{align}
        \int_\cone |\nabla u_k - \nabla u_{1, k} - \nabla u_{2, k}|^2 \ dx 
        & = \int_\cone |\nabla u_k|^2 - 2 \int_\cone \nabla u_k \nabla (u_{1, k} + u_{2, k}) \ dx \nonumber \\
        & \quad + \int_\cone |\nabla (u_{1, k} + u_{2, k})|^2 \ dx \nonumber \\
        & = 2 \int_\cone \lambda_1(\Omega_k) u_k (u_k - u_{1, k} - u_{2, k}) \ dx \nonumber \\
        & \quad - \int_\cone \left(|\nabla u_k|^2 - |\nabla (u_{1, k} + u_{2, k})|^2\right) \ dx \nonumber \\
        & \leq C \|u_k\|_{L^2(\cone)}\|u_k - u_{1, k} - u_{2, k}\|_{L^2(\cone)} \nonumber \\
        & \quad - \int_\cone \left(|\nabla u_k|^2 - |\nabla (u_{1, k} + u_{2, k})|^2\right) \ dx, \nonumber 
    \end{align}    
    where $C > 0$ is a fixed constant that does not depend on $k$ (since the sequence $\{\lambda_1(\Omega_k)\}_{k \in \N}$ is bounded). Since, by \eqref{eq:cone_dichotomy},
    \begin{equation*}
        0 \leq \lim_{k \to \infty} C \|u_k\|_{L^2(\cone)} \|u_k - u_{1, k} - u_{2, k}\|_{L^2(\cone)} = 0
    \end{equation*}
    and
    \begin{equation*}
        \limsup_{k \to \infty} \left(- \int_\cone \left(|\nabla u_k|^2 - |\nabla (u_{1, k} + u_{2, k})|^2\right) \right) \leq 0,
    \end{equation*}
    it follows that
    \begin{equation*}
        \frac{\displaystyle \int_\cone |\nabla (u_{1, k} + u_{2, k})|^2 \ dx}{\displaystyle \int_\cone (u_{1, k}^2 + u_{2, k}^2)} \to \lambda(\cone, m) \quad \text{ as } k \to \infty.
    \end{equation*}
    Now, recalling the inequality
    \begin{equation*}
        \frac{a + b}{c + d} \geq \min\left\{\frac{a}{c}, \frac{b}{d}\right\}, \quad \forall a, b, c, d > 0,
    \end{equation*}
    it follows that one of the sequences $\Omega_{i, k}$, $i = 1, 2$, is minimizing for $\eqref{eq:cone_min_A}$. Since $|\Omega_{i, k}| < m - \varepsilon$ for some $\varepsilon > 0$, the strict monotonicity \eqref{eq:cone_monotonicity} would allow us to find a quasi-open set $\Omega'$ of measure $m$ such that $\lambda_1(\Omega') < \lambda(\cone, m)$, a contradiction.

    We have thus proved that the vanishing and dichotomy cases cannot occur, so that compactness holds for our sequence $\{u_k\}_{k \in \N}$. Roughly speaking, \eqref{eq:cone_compactness} means that the mass of $u_k$ concentrates in a ball around $y_k$. As the next step shows, the same is true for the eigenvalue as well (in other words, the tails of $\Omega_k$ do not play an important role). 

    Precisely, we claim that there exists a constant $\Lambda > 0$ such that for every fixed $\varepsilon > 0$ there exist $R_\varepsilon > 1$ and $k_\varepsilon \in \N$, both depending only on $\varepsilon$, such that
    \begin{equation}
        \label{eq:cone_concentration_eigenvalue}
        \lambda_1(\Omega_k) \geq \lambda_1(B_{2R}(y_k) \cap \Omega_k) - \Lambda \varepsilon \quad \forall k \geq k_\varepsilon, \ \forall R \geq R_\varepsilon.
    \end{equation}
    To prove the claim, let us fix $\varepsilon > 0$ and let $R > 0$ be the corresponding radius given by the concentration-compactness theorem such that \eqref{eq:cone_compactness} holds. Let $\varphi$ be a fixed cut-off function with the following properties:
    \begin{align}
        & \varphi \in C_c^\infty(\R^N) \nonumber \\
        & 0 \leq \varphi \leq 1 \text{ in } \R^N, \nonumber \\
        & \varphi \equiv 1 \text{ in } B_R(O), \nonumber \\
        & \varphi \equiv 0 \text{ in } \R^N \setminus B_{2R}(O), \nonumber \\
        & |\nabla \varphi| \leq \frac{C_0}{R} \text{ in } \R^N, \nonumber 
    \end{align}
    where $C_0$ is a constant independent of $R$. Let $\varphi_k$ be the translation of $\varphi$ with center at $y_k$, that is, 
    \begin{equation*}
    \varphi_k(x) = \varphi(x - y_k), \quad x \in \R^N.
    \end{equation*}
    Now we note that
    \begin{align}
        \int_\cone |\nabla u_k|^2 \ dx
        & \geq \int_\cone |\nabla u_k|^2 \varphi_k^2 \ dx \nonumber \\
        & = \int_\cone |\nabla (u_k \varphi_k)|^2 \ dx - 2 \int_\cone (u_k \varphi_k) (\nabla u_k \cdot \nabla \varphi_k) \ dx - \int_\cone u_k^2 |\nabla \varphi_k|^2 \ dx. \nonumber
    \end{align}
    By Hölder's inequality, we have
    \begin{align}
        \int_\cone (u_k \varphi_k) (\nabla u_k \cdot \nabla \varphi_k) \ dx
        & \leq \|\varphi_k\|_\infty \|\nabla \varphi_k\|_\infty \int_\cone |u_k| |\nabla u_k| \ dx \nonumber \\
        & \leq \frac{C_0}{R} \|u_k\|_{L^2(\cone)} \|\nabla u_k\|_{L^2(\cone)} \nonumber \\
        & \leq \frac{C_0 C_1}{R}, \nonumber
    \end{align}
    where also $C_1$ is independent of $k$ and $R$. Similarly,
    \begin{equation*}
        \int_\cone u_k^2 |\nabla \varphi_k|^2 \ dx \leq \frac{C_0^2}{R^2}.
    \end{equation*}
    Without loss of generality, we can assume that $R > 1$. Hence
    \begin{equation*}
        \int_\cone |\nabla u_k|^2 \ dx \geq \int_\cone |\nabla (u_k \varphi_k)|^2 \ dx - \frac{C_2}{R},
    \end{equation*}
    where $C_2 > 0$ is independent of $k$ and $R$.
    Next, since $\varphi_k^2 \equiv 1$ in $B_R(y_k)$ and in view of \eqref{eq:cone_compactness}, we have:
    \begin{align}
        \int_\cone u_k^2 \ dx 
        & = \int_\cone u_k^2 \varphi_k^2 \ dx + \underbrace{\int_{\cone \cap B_R(y_k)} u_k^2 (1 - \varphi_k^2) \ dx}_{ = 0} + \int_{\cone \setminus B_R(y_k)} u_k^2 (1- \varphi_k^2) \ dx \nonumber \\
        & \leq \int_\cone (u_k \varphi_k)^2 \ dx + \varepsilon. \nonumber
    \end{align}
    Since $u_k \varphi_k \in H_0^1(B_{2R}(y_k) \cap \Omega_k; \cone)$, it is an admissible test function for the Rayleigh quotient. We have:
    \begin{align}
        \lambda_1(B_{2R}(y_k) \cap \Omega_k)
        & \leq \frac{\displaystyle \int_\cone |\nabla (u_k \varphi_k)|^2 \ dx}{\displaystyle \int_\cone (u_k \varphi_k)^2 \ dx} \nonumber \\
        & \leq \frac{\displaystyle \int_\cone |\nabla u_k|^2 \ dx + \frac{C_2}{R}}{1 - \varepsilon}. \nonumber 
    \end{align}
    Since $C_2$ is independent of $k$ and $R$, we can assume that $\mfrac{C_2}{R} < \varepsilon$ (up to taking a larger $R$). Then, by a simple Taylor series expansion, we obtain
    \begin{align}
        \lambda_1(B_{2R}(y_k) \cap \Omega_k) \leq \lambda_1(\Omega_k) + (\lambda_1(\Omega_k) + 1) \varepsilon + o(\varepsilon). \nonumber
    \end{align}
    Choosing $\Lambda = \lambda(\cone, m) + K$, for a large constant $K > 0$, we readily obtain \eqref{eq:cone_concentration_eigenvalue}.

    Our next step is to show that the sequence of points $\{y_k\}_{k \in \N} \subset \cone$ is bounded. The idea is that if we assume that the sequence is unbounded, using the previous step and the fact that the cone ``flattens out" at infinity, we converge in some sense to the eigenvalue of a half-ball, contradicting the assumption \eqref{eq:cone_ineq}.

    For the sake of contradiction, let us assume that there exists a subsequence (still indexed by $k$) such that
    \begin{equation*}
        |y_k| \to + \infty \quad \text{ as } k \to \infty.
    \end{equation*}
    By our key assumption \eqref{eq:cone_ineq}, there exists $\varepsilon > 0$ sufficiently small such that
    \begin{equation}
        \label{eq:cone_contradiction_assumption_y_k}
        \lambda(\cone, m) + \Lambda \varepsilon < \lambda(\Sigma_\sphereplus, m),
    \end{equation}
    where $\Lambda$ is the constant from \eqref{eq:cone_concentration_eigenvalue}. Moreover, by \eqref{eq:cone_concentration_eigenvalue} there exists $R_\varepsilon > 0$ sufficiently large, depending only on $\varepsilon$, such that for all sufficiently large $k$ it holds
    \begin{equation}
        \label{eq:cone_ineq_y_k}
        \lambda_1(\Omega_k) \geq \lambda_1(\BTR \cap \Omega_k) - \Lambda \varepsilon. 
    \end{equation}
    Note that there exists $\bar k \in \N$ such that $d(y_k, \pc) \leq 2R$ for all $k \geq \bar k$, because otherwise we would have
    \begin{equation*}
        \lambda_1(B_{2R}(y_k) \cap \Omega_k) \geq \lambda_1(B^m) > \lambda(\Sigma_\sphereplus, m), 
    \end{equation*}
    where $B^m$ is the ball of measure $m$, contradicting \eqref{eq:cone_ineq_y_k}-\eqref{eq:cone_concentration_eigenvalue}. Therefore there exists a sequence $\{z_k\}_{k \in \N} \subset \pc$ such that for all $k \geq \bar k$ it holds
    \begin{align}
        & z_k \in \left(\overline{\BTR \cap \Omega_k} \cap \pc\right), \nonumber \\
        & \BTR \subset \BFR. \nonumber
    \end{align}
    Since the functional $\Omega \mapsto \lambda_1(\Omega)$ is monotone with respect to inclusion and $H_0^1(\BTR \cap \Omega_k; \cone) \subset H_0^1(\BFR \cap \Omega_k; \cone)$, we have that
    \begin{equation}
        \label{eq:cone_ineq_monot}
        \lambda_1(\BTR \cap \Omega_k) \geq \lambda_1(\BFR \cap \Omega_k) \quad \forall k \geq \bar k.
    \end{equation}
    We now show that, up to a further subsequence still indexed by $k$, it holds
    \begin{equation}
        \label{eq:cone_claim}
        \lambda_1(\BFR \cap \Omega_k) \geq \lambda(\Sigma_\sphereplus, m) +  o(1) \quad \text{ as } k \to \infty. 
    \end{equation}
    To prove the claim \eqref{eq:cone_claim}, we use the geometric construction and considerations presented in the proof of Proposition \ref{prop:cone_bound_lambda_hemisphere}. We set
    \begin{equation*}
        q_k \coloneqq \frac{z_k}{|z_k|} \in \partial D, \quad k \in \N.
    \end{equation*}
    By the compactness of $\partial D$ in $\sphere$, there exists $q \in \partial D$ such that $q_k \to q$ (up to a subsequence still indexed by $k$), where the convergence is to be understood with respect to the geodesic distance on $\sphere$. As in the proof of Proposition \ref{prop:cone_bound_lambda_hemisphere}, we take $U \subset T_q\pc$ to be a convex neighborhood of the origin in the tangent plane and a map $g: \overline{U} \to \R$ such that $g(0) = 0$, $\nabla_{x'} g(0) = 0$, and the graph of $g$ on $U$ is a neighborhood $V$ of $q$ on $\pc$. As before, we use the coordinates $(x', x_N)$ with respect to the orthonormal frame $\{v_1, \ldots, v_{N - 1}, - \nu(q)\}$, where $\{v_i\}_{i = 1, \ldots, N - 1}$ is an orthonormal frame of $T_q\pc$. Let $\Pi: \R^N \to T_q\pc$ be the orthogonal projection operator. Since $q_k \to q$, then for all sufficiently large $k$ we have
    \begin{align}
        & q_k \in V, \nonumber \\
        & \Pi(q_k - q) \in U, \nonumber
    \end{align}
    and moreover
    \begin{equation*}
        \Pi(q_k - q) \to 0 \quad \text{ as } k \to \infty.
    \end{equation*}
    For each $k \in \N$, let $\bar x_k'$ be the coordinates of $\Pi(q_k - q)$:
    \begin{equation*}
        \bar x_k' = (\bar x_{1, k}', \ldots, \bar x_{N - 1, k}') = \Pi(q_k - q) \in T_q\pc. 
    \end{equation*}
    In addition, let us set, for each $k \in \N$, the translated domain
    \begin{equation*}
        U_k \coloneqq U - \bar x_k' \nonumber
    \end{equation*}
    and the function $g_k: U_k \to \R$ given by
    \begin{equation*}
        g_k(x') = g(x' + \bar x_k') - g(\bar x_k'), \quad x' \in U_k.
    \end{equation*}
    It is clear that $V - q_k$ is the graph of $g_k$. Now, by construction, $\bar x_k' \to 0$ in $T_q\pc$ as $k \to \infty$, and therefore there exists a ball $B_{\bar R} = B_{\bar R}(0)$ in $T_q\pc$ such that $B_{\bar R} \subset U_k$ for all sufficiently large $k$. Let us set $t_k = |z_k|$ for $k \in \N$. By definition, $t_k \to + \infty$ as $k \to \infty$, and therefore $t_k U_k$ covers $T_q\pc$ as $k \to \infty$. Next, for $k \in \N$ we define the scaled map
    \begin{align}
        h_k(x') = t_k g_k \left(\frac{x'}{t_k}\right) = t_k \left(g\left(\frac{x'}{t_k} + \bar x_k' \right) - g(\bar x_k') \right), \quad x' \in t_k U_k. \nonumber
    \end{align}
    Arguing as in \eqref{eq:cone_mean_value_theorem}, we obtain
    \begin{align}
        \frac{\partial h_k}{\partial x_i'}(x')
        & = \frac{\partial g}{\partial x_i'}\left(\frac{x'}{t_k} + \bar x_k' \right) \nonumber \\
        & = \frac{\partial g}{\partial x_i'}(\bar x_k') +  \left(\nabla_{x'} \frac{\partial g}{\partial x_i'} \right) (\eta_k + \bar x_k') \cdot \frac{x'}{t_k}, \nonumber
    \end{align}
    where $\eta_k$ belongs to the segment joining $\bar x_k'$ and $\bar x_k' + \mfrac{x'}{t_k}$, for every $k \in \N$. Consider a ball $\btr$ in $T_q\pc$, where $\widetilde R$ is to be chosen later and independently of $k$. Since $t_k U_k$ covers $T_q\pc$ as $k \to \infty$, then $\btr \subset t_k U_k$ for all sufficiently large $k$. Now, since $g$ is smooth, $\bar x_k' \to 0$ in $T_q\pc$ as $k \to \infty$ and $\nabla_{x'} g(0) = 0$, arguing as in \eqref{eq:cone_nabla_xprime_g_t_vanishes} we obtain that
    \begin{equation*}
        \lim_{k \to \infty} \max_{x' \in \overline{\btr}} |\nabla_{x'} h_k(x')| = 0.
    \end{equation*}
    Recalling the notations
    \begin{align}
        & C_{\widetilde R}^+ \coloneqq \{(x', x_N) \in \R^N \ : \ x' \in \btr, \ x_N > 0\}, \nonumber \\
        & E_k^+ \coloneqq \{(x', x_N) \in \R^N \ : \ x' \in \btr, \ x_N > h_k(x')\}, \nonumber 
    \end{align}
    we consider the diffeomorphism 
    \begin{equation*}
        F_k: \overline{C_{\widetilde R}^+} \to \overline{E_k^+}
    \end{equation*}
    given by
    \begin{equation*}
        F_k(x', x_N) = (x', x_N + h_k(x')), \quad (x', x_N) \in \overline{C_{\widetilde R}^+}.
    \end{equation*}
    Note that for all sufficiently large $k$ it holds
    \begin{equation*}
        \left(\overline{\BFR \cap \Omega_k} \cap \pc\right) - z_k \subset t_k(V - q_k).
    \end{equation*}
    For simplicity, we set
    \begin{equation*}
        \ork \coloneqq (\BFR \cap \Omega_k) - z_k.
    \end{equation*}
    Since $\ork$ is uniformly bounded and $t_kU_k$ covers $T_q\pc$ as $k \to \infty$, then there exists $\widetilde R > 0$ independent of $k$ such that
    \begin{equation*}
        F_k^{-1}(\overline{\ork}) \subset \btr \times [0, + \infty).
    \end{equation*}
    We denote by $\tuk \in H_0^1(\ork; \cone - z_k)$ the first positive $L^2$-normalized eigenfunction of $\ork$, and set $\TUK \coloneqq \tuk \circ F_k$. It then readily follows that $\TUK \in H_0^1(F_k^{-1}(\ork); \R^N_+)$. As in the proof of Proposition \ref{prop:cone_bound_lambda_hemisphere}, we have that
    \begin{align}
        & |F_k^{-1}(\ork)| = |\ork| \leq m, \nonumber \\
        & \int_{\ork} \tuk^2 \ dx = \int_{F_k^{-1}(\ork)} \TUK^2 \ dx = 1, \nonumber \\
        & \int_{F_k^{-1}(\ork)} |\nabla \TUK|^2 \ dx \leq M_k^2 \int_{\ork} |\nabla \tuk|^2 \ dx, \nonumber
    \end{align}
    where
    \begin{equation*}
        M_k = \max_{(x', x_N) \in \overline{C_{\widetilde R}^+}} \max_{|\zeta| \leq 1} |JF_k(x', x_N) \zeta| \to 1 \quad \text{ as } k \to \infty.
    \end{equation*}
    Now, in view of Remark \ref{rem:cone_sphere} and the monotonicity of the eigenvalue with respect to a scaling of the domain, we have that
    \begin{align}
        \lambda_1(\ork)
        & \geq \frac{1}{M_k^2} \int_{F_k^{-1}(\ork)} |\nabla \TUK|^2 \ dx \nonumber \\
        & \geq \frac{1}{1 + o(1)} \lambda_1(F_k^{-1}(\ork)) \nonumber \\
        & \geq \frac{1}{1 + o(1)} \lambda(\Sigma_\sphereplus, |\ork|) \nonumber \\
        & \geq \frac{1}{1 + o(1)} \lambda(\Sigma_\sphereplus, m) \nonumber \\
        & = \lambda(\Sigma_\sphereplus, m) + o(1) \nonumber
    \end{align}    
    as $k \to \infty$, which is precisely \eqref{eq:cone_claim}. Now, combining \eqref{eq:cone_ineq_y_k}, \eqref{eq:cone_ineq_monot}, \eqref{eq:cone_claim} we obtain
    \begin{equation*}
        \lambda_1(\Omega_k) \geq \lambda(\Sigma_\sphereplus, m) + o(1) - \Lambda \varepsilon \quad \text{ as } k \to \infty,
    \end{equation*}
    contradicting \eqref{eq:cone_contradiction_assumption_y_k}. It therefore follows that the sequence $\{y_k\}_{k \in \N} \subset \cone$ is bounded.

    The boundedness of the sequence $\{y_k\}_{k \in \N}$ allows us to prove that the sequence of eigenfunctions $\{u_k\}_{k \in \N}$ admits a strongly convergent subsequence in $L^2(\cone)$.

    To the end of proving this claim, we begin by showing that 
    \begin{equation}
        \label{eq:cone_exterior_vanishing}
        \lim_{R \to \infty} \sup_{k \in \N} \int_{\cone \setminus B_R(O)} u_k^2 \ dx = 0.
    \end{equation}
    We proceed by contradiction. Should \eqref{eq:cone_exterior_vanishing} not hold, there would exist $\varepsilon' > 0$ and a sequence $\{R_j\}_{j \in \N} \subset \R^+$, with $R_j \to + \infty$ as $j \to \infty$, and a subsequence $\{u_j\}_{j \in \N}$ of eigenfunctions such that
    \begin{equation*}
        \int_{\cone \setminus B_{R_j}(O)} u_j^2 \ dx \geq \frac{\varepsilon'}{2} \quad \forall j \in \N.
    \end{equation*}
    On the other hand, recalling \eqref{eq:cone_compactness}, if we take $\varepsilon = \mfrac{\varepsilon'}{4}$ and $R_\varepsilon$ be the corresponding radius given by \eqref{eq:cone_compactness}, then for all $k \in \N$ we have
    \begin{equation*}
        \int_{B_{R_\varepsilon}(y_k) \cap \cone} u_k^2 \ dx \geq 1 - \frac{\varepsilon'}{4}.
    \end{equation*}
    Now we make use of the boundedness of the sequence $\{y_k\}_{k \in \N}$. It implies that there exists $\bar R > 0$ independent of $k$ such that $B_{R_\varepsilon}(y_k) \subset B_{\bar R}(O)$ for all $k \in \N$. Hence
    \begin{equation*}
        \int_{B_{\bar R}(O) \cap \cone} u_k^2 \ dx \geq \int_{B_{R_\varepsilon}(y_k) \cap \cone} u_k^2 \ dx \geq 1 - \frac{\varepsilon'}{4}
    \end{equation*}
    for all $k$ sufficiently large. However, since $R_j \to + \infty$, there exists $j_0 \in \N$ such that $R_j \geq \bar R$ for all $j \geq j_0$, and the corresponding eigenfunctions satisfy
    \begin{align}
        \int_\cone u_j^2 \ dx 
        & = \int_{B_{R_j}(O) \cap \cone} u_j^2 \ dx + \int_{\cone \setminus B_{R_j}(O)} u_j^2 \ dx \nonumber \\
        & \geq 1 - \frac{\varepsilon'}{4} + \frac{\varepsilon'}{2} \nonumber \\
        & > 1, \nonumber
    \end{align}
    a contradiction with the constraint
    \begin{equation*}
        \int_\cone u_k^2 \ dx = 1 \quad \forall k \in \N. 
    \end{equation*}
    This proves \eqref{eq:cone_exterior_vanishing}, which in turn implies that for any $\varepsilon > 0$ there exists $R > 0$ such that
    \begin{equation*}
        \int_{\cone \setminus B_R(O)} u_k^2 \ dx < \varepsilon \quad \forall k \in \N.
    \end{equation*}
    Consider the truncated functions
    \begin{equation*}
        f_k \coloneqq \chi_{B_R(O)} u_k, \quad k \in \N,
    \end{equation*}
    where $\chi_E$ denotes, as usual, the characteristic function of the set $E$. We readily obtain that the sequence $\{f_k\}_{k \in \N}$ is bounded in $H^1(B_R(O) \cap \cone)$, because $\{u_k\}_{k \in \N}$ is bounded in $H^1(\cone)$, and therefore there exists $f \in H^1(B_R(O) \cap \cone)$ such that
    \begin{equation*}
        f_k \rightharpoonup f \quad \text{ weakly in } H^1(B_R(O) \cap \cone) \quad \text{ as } k \to \infty.
    \end{equation*} 
    In addition, it is also immediate from \eqref{eq:cone_exterior_vanishing} that
    \begin{equation}
        \label{eq:cone_uk_vk}
        \|u_k - f_k\|_{L^2(\cone)} \leq \varepsilon \quad \forall k \in \N.
    \end{equation}
    Moreover, since $B_R(O) \cap \cone$ is bounded and satisfies the cone condition, the embedding
    \begin{equation*}
        H^1(B_R(O) \cap \cone) \hookrightarrow L^2(B_R(O) \cap \cone)
    \end{equation*}
    is compact, and therefore $f_k \to f$ strongly in $L^2(B_R(O) \cap \cone)$, and hence in $L^2(\cone)$ (up to a subsequence). This relative compactness of $\{f_k\}_{k \in \N}$, together with \eqref{eq:cone_uk_vk}, imply that the sequence $\{u_k\}_{k \in \N}$ can be covered with finitely many balls of radius $\varepsilon$. Since $\varepsilon$ is arbitrary, it follows (by definition) that $\{u_k\}_{k \in \N}$ is totally bounded in $L^2(\cone)$. Since $L^2(\cone)$ is a Banach space, then it follows that $\{u_k\}_{k \in \N}$ is relatively compact (\cite[Chapter 7]{Munkres2000}), i.e., there exists $u \in L^2(\cone)$ such that
    \begin{equation*}
        u_k \to u \quad \text{ strongly in } L^2(\cone) \quad \text{ as } k \to \infty. 
    \end{equation*}

    We now come to the conclusion of the proof. Since the sequence $\{u_k\}_{k \in \N}$ is bounded in $H^1(\cone)$, it converges weakly in this space (up to a subsequence). By the previous step, $u_k \to u$ strongly in $L^2(\cone)$, hence $u$ is precisely the weak limit of $u_k$ in $H^1(\cone)$ as $k \to \infty$, since the embedding $H^1(\cone) \hookrightarrow L^2(\cone)$ is a continuous linear map. It is clear that $u \geq 0$ in $\cone$. We set
    \begin{equation*}
        \Omega \coloneqq \{u > 0\}.
    \end{equation*}
    This set is a quasi-open subset of $\cone$ and, by Fatou's Lemma,
    \begin{equation*}
        |\Omega| = \int_\cone \chi_{\{u > 0\}} \leq \liminf_{k \to \infty} \int_\cone \chi_{\Omega_k} \ dx = m.
    \end{equation*}
    In addition,
    \begin{equation*}
        \|u\|_{H^1(\cone)} \leq \liminf_{k \to \infty} \|u_k\|_{H^1(\cone)}, 
    \end{equation*}
    which, combined with $\|u_k\|_{L^2(\cone)} \to \|u\|_{L^2(\cone)}$ as $k \to \infty$ leads to
    \begin{equation*}
        \lambda_1(\Omega) \leq \frac{\displaystyle \int_\Omega |\nabla u|^2 \ dx}{\displaystyle \int_\Omega u^2 \ dx} \leq \liminf_{k \to \infty} \frac{\displaystyle \int_\cone |\nabla u_k|^2 \ dx}{\displaystyle \int_\cone u_k^2 \ dx} = \liminf_{k \to \infty} \lambda_1(\Omega_k) = \lambda(\cone, m).
    \end{equation*}
    Should $|\Omega| < m$, we would be able to consider a suitable scaling $t\Omega$ such that $|t\Omega| = m$, in which case we would have $\lambda_1(t\Omega) < \lambda(\cone, m)$, absurd. Hence $|\Omega| = m$ and
    \begin{equation*}
        \lambda_1(\Omega) = \lambda(\cone, m),
    \end{equation*}
    which completes the proof.
\end{proof}

To conclude, we show that the assumption \eqref{eq:cone_ineq} is not empty.

\begin{prop}
    \label{prop:cone_measure_condition}
    Let $D \subset \sphere$ be a smooth domain such that 
    \begin{equation}
        \label{eq:cone_ineq_Hausdorff}
        \haus(D) < \haus(\sphereplus).
    \end{equation}
    Then $\lambda(\cone, m)$ is attained, for any $m > 0$.
\end{prop}

\begin{proof}
    Fix $m = |\Omega_D|$ and let $R_m$ be the radius of the ball such that
    \begin{equation*}
        |B_{R_m}(O)| = 2m.
    \end{equation*}
    By assumption \eqref{eq:cone_ineq_Hausdorff}, $R_m < 1$. Moreover, since the first eigenfunction in the ball is radial,
    \begin{equation*}
        \lambda_1(B_{R_m}(O)) = \lambda(\Sigma_\sphereplus, m) = \lambda_1(B_{R_m}(O) \cap \cone) > \lambda_1(\Omega_D),
    \end{equation*}
    since $(B_{R_m}(O) \cap \cone) \subset \Omega_D$ and by the scaling property of eigenvalues. Then
    \begin{equation*}
        \lambda(\cone, m) < \lambda(\Sigma_\sphereplus, m)
    \end{equation*}
    and we conclude by using Theorem \ref{thm:cone_existence} together with Remark \ref{rem:cone_scaling}.
\end{proof}

\section{Qualitative properties of minimizers}
\label{sec:cone_bound_conn_reg}

In this section, we study topological and regularity properties of minimizers for the problem of minimizing the first eigenvalue of $- \Delta$ with mixed boundary conditions:
\begin{equation}
    \label{eq:cone_min_B}
    \inf\{\lambda_1(\Omega) \ : \ \Omega \subset \cone \text{ is quasi-open}, \ |\Omega| = m\},
\end{equation}
where $m > 0$ is given.

\begin{remark}
    Recall that since the cone $\cone$ is invariant under scaling, it holds
    \begin{equation}
    \label{eq:cone_monotonicity_b}
    \lambda_1(t\Omega) = \frac{\lambda_1(\Omega)}{t^2}, \quad t\Omega = \{tx \in \R^N \ : \ x \in \Omega\}.
    \end{equation}
    Moreover, in view of Theorem \ref{thm:cone_existence}, we know that at least one solution for \eqref{eq:cone_min_B} exists (provided that \eqref{eq:cone_ineq} holds), for whichever $m$ is chosen. As is well known, \eqref{eq:cone_min_B} is equivalent to
    \begin{equation*}
        \inf\{ \lambda_1(\Omega)|\Omega|^{\frac{2}{N}} \ : \ \Omega \subset \cone \text{ is quasi-open}\},
    \end{equation*}
    because of \eqref{eq:cone_monotonicity_b} and since $|t\Omega| = t^N |\Omega|$.
\end{remark}

The next result is a well-known property of minimizers for spectral problems in $\R^N$, which also holds in $\cone$ because $\cone$ is invariant under scaling. Since it is fundamental for the proofs of the qualitative properties in this section, we report it here for the convenience of the reader.

\begin{lemma}[{\cite{Bucur2012, BucurFreitas2017}}]
    \label{lem:cone_equivalent_problems}
    For every $\tau > 0$, problem \eqref{eq:cone_min_B} is equivalent to
    \begin{equation}
        \label{eq:cone_min_equiv}
        \inf\{\lambda_1(\Omega) + \tau |\Omega| \ : \ \Omega \subset \cone \text{ is quasi-open}\},
    \end{equation}
    in the sense that if $\Omega$ is a solution to \eqref{eq:cone_min_B} (respectively, \eqref{eq:cone_min_equiv}), then there exists a scaling of $\Omega$ which is a solution for \eqref{eq:cone_min_equiv} (respectively, \eqref{eq:cone_min_B}).
\end{lemma}

\begin{proof}
    For any fixed quasi-open set $\Omega \subset \cone$ and any given $\tau > 0$, there exists an optimal scaling factor $\bar t$ that minimizes the function
    \begin{equation*}
        t \in (0, + \infty) \mapsto G(t\Omega) = t^{-2}\lambda_1(\Omega) + t^N \tau |\Omega|.
    \end{equation*}
    Indeed, $G(t\Omega)$ is smooth, $G(t\Omega) \to + \infty$ as $t \to 0^+$ and $G(t\Omega) \to + \infty$ as $t \to + \infty$. Since
    \begin{equation*}
        \frac{d}{dt} G(t \Omega) = - 2 t^{-3} \lambda_1(\Omega) + N t^{N - 1} \tau |\Omega|,
    \end{equation*}
    then
    \begin{equation*}
        \frac{d}{dt}G(t\Omega) = 0
    \end{equation*}
    has the unique root
    \begin{equation*}
        \bar t_\Omega = \left(\frac{2 \lambda_1(\Omega)}{N \tau |\Omega|}\right)^{\frac{1}{N + 2}}.
    \end{equation*}

    Now, note that
    \begin{align}
        G(\bar t \Omega) 
        & = \underbrace{\left(\left(\frac{N \tau}{2}\right)^{\frac{2}{N + 2}} + \tau^{\frac{2}{N + 2}}\left(\frac{2}{N}\right)^{\frac{N}{N + 2}}\right)}_{\coloneqq C_{N, \tau}} \lambda_1(\Omega)^{\frac{N}{N + 2}} |\Omega|^{\frac{2}{N + 2}} \nonumber \\
        & = C_{N, \tau} \left(\lambda_1(\Omega) |\Omega|^{\frac{2}{N}}\right)^{\frac{N}{N + 2}}. \nonumber
    \end{align}
    
    Since the map
    \begin{equation*}
        s \in (0, + \infty) \mapsto C_{N, \tau} s^{\frac{N}{N + 2}}
    \end{equation*}
    is strictly increasing, then minimizing $G(\bar t_\Omega \Omega)$ among all quasi-open sets $\Omega \subset \cone$ is equivalent to \eqref{eq:cone_min_B}. Consequently, if a quasi-open set $\Omega$ minimizes \eqref{eq:cone_min_B}, then a suitable rescaling of $\Omega$ minimizes \eqref{eq:cone_min_equiv}. On the other hand, if $\Omega$ is a minimizer for \eqref{eq:cone_min_equiv}, then $\bar t_\Omega = 1$. In particular, $\Omega$ is a minimizer for \eqref{eq:cone_min_B} with $m = |\Omega|$. Then, for $m \neq |\Omega|$, a simple scaling of $\Omega$ yields a minimizer for \eqref{eq:cone_min_B}.
\end{proof}

\begin{remark}
    As a consequence of Lemma \ref{lem:cone_equivalent_problems}, the solutions of \eqref{eq:cone_min_B} will have the same boundedness, topological, and regularity properties as the solutions of \eqref{eq:cone_min_equiv}, since they are obtained by simply scaling the domain, which preserves all these properties.
\end{remark}

\subsection{Boundedness}

The idea is to adapt the arguments of \cite{Bucur2012} (see also \cite[Chapter 4]{Velichkov2015}) to the relative setting of the cone. Namely, we will show that any minimizer $\Omega$ for \eqref{eq:cone_min_B} is a shape subsolution for the energy, according to Definition \ref{def:cone_shape_subsolution}, which will imply the boundedness of $\Omega$, as well as the finite perimeter; see Proposition \ref{prop:cone_boundedness}.

The proofs are essentially the same as in \cite{Bucur2012}; hence, we will outline the main differences, omitting the details.

\begin{definition}
    \label{def:cone_shape_subsolution}
    A quasi-open set $\Omega \subset \cone$ is said to be a local shape subsolution for the (torsional) energy if there exists $\delta > 0$ and a constant $\Lambda > 0$ such that for every quasi-open subset $\omega \subset \Omega$ for which $d_\gamma(\Omega, \omega) < \delta$, it holds
    \begin{equation*}
        E(\Omega) + \Lambda |\Omega| \leq E(\omega) + \Lambda |\omega|
    \end{equation*}
    where $d_\gamma$ is the distance on the class of quasi-open sets for the topology of $\gamma$-convergence (see \eqref{eq:cone_def_d_gamma}).
\end{definition}

We begin with the following analogue of \cite[Lemma 3]{Bucur2012}.

\begin{lemma}
    \label{lem:cone_continuity_eigenvalues}
    Let $\Omega \subset \cone$ be a quasi-open set of finite measure. For every $k \in \N$, there exists a constant $c_k(\Omega)$ depending only on $\Omega$ such that for every $j \leq k$ and every quasi-open subset $\omega \subset \Omega$ it holds that
    \begin{equation*}
        \left|\frac{1}{\lambda_j(\Omega)} - \frac{1}{\lambda_j(\omega)} \right| \leq c_k(\Omega) d_\gamma(\Omega, \omega).
    \end{equation*}
\end{lemma}

\begin{proof}
    The proof relies on properties of the resolvent operator for the Dirichlet problem and $L^\infty$ bounds for eigenfunctions. 
    
    As already explained in Section \ref{sec:cone_prelim}, the resolvent operator for the mixed boundary value problem is akin to the resolvent for the pure Dirichlet problem. Moreover, they give analogous uniform $L^\infty$ bounds; see \cite[Remark 2.6 and Proposition 2.7]{ButtazzoVelichkov2016}.
\end{proof}

It then follows that every solution of \eqref{eq:cone_min_equiv} is a shape subsolution for the energy.

\begin{prop}
    \label{prop:cone_energy_subsolution}
    Let $\Omega \subset \cone$ be a solution of \eqref{eq:cone_min_equiv}. Then $\Omega$ is a shape subsolution for the torsional energy.
\end{prop}

\begin{proof}
    As in \cite[Theorem 2]{Bucur2012}, it is a consequence of Lemma \ref{lem:cone_equivalent_problems} (with $\tau = 1$) and Lemma \ref{lem:cone_continuity_eigenvalues}.
\end{proof}

\begin{prop}
    \label{prop:cone_boundedness}
    Let $\Omega \subset \cone$ be a minimizer for \eqref{eq:cone_min_equiv}. Then $\Omega$ is bounded and has finite perimeter.
\end{prop}

\begin{proof}
    In view of Proposition \ref{prop:cone_energy_subsolution}, it suffices to show that local shape subsolutions for the torsional energy are bounded and have finite perimeter. The argument is the same as in \cite{Bucur2012}, so we just outline some simple necessary technical modifications. 
    
    Clearly, one is to work in the space $H_0^1(\Omega; \cone)$. Next, for an analogue of \cite[Lemma 1]{Bucur2012}, it suffices to consider the auxiliary functions $v_r$ with the appropriate homogeneous Neumann condition on $\pc$. This will guarantee that the test functions belong to the correct spaces, as well as the correct integration by parts when necessary. 
    
    Note that the boundary trace theorem in $W^{1, 1}$, as well as the coarea formula, are available in the relative setting (see \cite{Giusti1984}).
\end{proof}

\subsection{Topological properties}
We begin by proving that the minimizers are open, and the corresponding first eigenfunctions are Lipschitz continuous. To this end, we adapt the argument of \cite[Proposition 5.4]{Velichkov2015}, which in turn relies on adapting a well-known result of Alt and Caffarelli (\cite{AltCaffarelli1981}, see also \cite{BrianconHayouniPierre2005} and \cite[Lemma 5.1]{Velichkov2015}). 

\begin{prop}
    \label{prop:cone_open}
    Let $\Omega \subset \cone$ be a minimizer for \eqref{eq:cone_min_equiv}. Then $u_\Omega$ is locally Lipschitz continuous inside $\cone$ and $\Omega$ is open in $\cone$.
\end{prop}
\begin{proof}
    Begin by noting that \cite[Lemma 5.1]{Velichkov2015} works, with obvious small modifications, in the functional space $H^1(\cone)$. So does \cite[Proposition 5.4]{Velichkov2015}. In particular, note that if $\Omega$ is a minimizer for \eqref{eq:cone_min_equiv}, then it is also a minimizer for \cite[(5.10)]{Velichkov2015} (with the appropriate definition for the eigenvalue, in the correct functional space).
\end{proof}

\begin{prop}
    \label{prop:cone_connectedness}
    Let $\Omega$ be a minimizer for \eqref{eq:cone_min_B}. Then $\Omega$ is connected.
\end{prop}
\begin{proof}
    We know that $\Omega$ is open (from Proposition \ref{prop:cone_open}). Should $\Omega$ have more than one connected component, then by a suitable scaling of one of the components (whose first eigenvalue is precisely the minimum for \eqref{eq:cone_min_B}), we would be able to find a set which satisfies the measure constraint and which has a smaller eigenvalue, a contradiction. Hence $\Omega$ is connected.
\end{proof}

\subsection{Regularity of the free boundary}

Following \cite{BriançonLamboley2009}, we have the following regularity result regarding the free boundary $\Gz$ for the minimizers $\Omega$ of \eqref{eq:cone_min_B}.

\begin{prop}
    \label{prop:cone_regularity}
    Let $\Omega \subset \cone$ be a minimizer for \eqref{eq:cone_min_B}. Then
    \begin{equation*}
        \Gz = \Gzreg \cup \GZsing,
    \end{equation*}
    where $\Gzreg$ is the reduced part of $\Gz$ (see, e.g., \cite{Maggi2012, EvansGariepy2015}) and $\GZsing$ is a singular part of the boundary, such that
    \begin{enumerate}[label=(\roman*)]
        \item $\haus(\Gz \setminus \Gzreg) = 0$;

        \item $\Gzreg$ is a real-analytic hypersurface in $\cone$;

        \item there exists $\Lambda > 0$ such that
        \begin{equation}
            \label{eq:cone_overdet_distributions}
            \Delta u_\Omega + \lambda_1(\Omega) u_\Omega = \sqrt{\Lambda} \haus\lfloor\Gzreg
        \end{equation}
        in the sense of distributions in $\cone$, where $\haus\lfloor\Gzreg$ denotes the restriction of the Hausdorff measure to $\Gzreg$.
    \end{enumerate}
\end{prop}

\begin{proof}
    Since the arguments are essentially those carried out in \cite{BriançonLamboley2009}, which are mostly of a local nature, we only outline the small necessary modifications.

    First, let us observe that we know that the minimizer $\Omega \subset \cone$ for \eqref{eq:cone_min_B} is bounded, by Proposition \ref{prop:cone_boundedness}. Hence there exists $R > 0$ large enough so that instead of looking at the whole space $H^1(\cone)$, we can fix the Sobolev space $H_0^1(B_R(O); \cone)$, for some $R > 0$ large enough such that $\Omega \subset B_R(O)$, define our eigenvalue problems accordingly (that is, for a domain $\omega \subset (B_R(O) \cap \cone)$, we define the eigenvalue problem as in  \eqref{eq:cone_eigenvalue_problem} but with the homogeneous Dirichlet condition also on $\partial \omega \cap \partial B_R(O)$), and $\Omega$ will be a minimizer for the first eigenvalue also in this restricted setting. Let us remark that, in this way, all integrations by parts work, and moreover, since $H_0^1(B_R(O); \cone) \hookrightarrow L^2(B_R(O))$ with compact injection, the arguments relying on the convergence of bounded sequences in $H_0^1(B_R(O); \cone)$ can be carried out.
\end{proof}

\begin{remark}
    As shown in \cite{LamboleySicbaldi2014} by an asymptotic expansion of the eigenfunction near the corner, critical shapes must touch the boundary of $\cone$ orthogonally.
\end{remark}

\section{Back to the overdetermined problem}
\label{sec:cone_conclusion}
In view of the integration by parts formula (see, e.g., \cite[Theorem 2.10]{Giusti1984}), \eqref{eq:cone_overdet_distributions} means that $u_\Omega$ satisfies the overdetermined condition (recall \eqref{eq:cone_overdet_pde}) in a sort of distributional/measure-theoretic sense. We will now show that \eqref{eq:cone_overdet_pde} holds in the classical sense on the regular part $\Gzreg$ of the boundary $\Gz$.

\begin{theorem}
    \label{thm:cone_overdet_minimizer}
    Let $\Omega \subset \cone$ be a minimizer for \eqref{eq:cone_min_B}. Then it holds that
    \begin{equation}
        \label{eq:cone_overdet_reg}
        \frac{\partial u_\Omega}{\partial \nu} = - \sqrt{\dfrac{2 \lambda_1(\Omega)}{N |\Omega|}} \quad \text{ on } \Gzreg.
    \end{equation}
\end{theorem}

\begin{proof}
    We may argue as in the proof of \cite[Proposition 7.4]{IacopettiPacellaWeth2022}.

    Fix $x_0 \in \Gzreg$ and let $B_r(x_0)$ be a small ball such that
    \begin{equation*}
        B_{2r}(x_0) \subset \cone, \quad (B_{2r}(x_0) \cap \Gz) \subset \Gzreg.
    \end{equation*}
    Let $\varphi \in C_c^\infty(B_r(x_0))$ and $\bar \nu$ be an extension of the normal vector $\nu$ on $\Gz$ to a smooth vector field defined in $\overline{B_r(x_0)}$ (see, e.g., \cite[Section 5.4]{HenrotPierre2018} for the construction). Then we consider the smooth vector field
    \begin{equation*}
        V = \varphi \bar \nu.
    \end{equation*}
    By construction, $V: \R^N \to \R^N$ is a smooth vector field with support in $B_r(x_0)$, $V(O) = 0$ and $V(x) = 0 \in T_x\pc$ for $x \in \pc \setminus\{O\}$. Then we consider the induced deformations
    \begin{equation*}
        \Omega_t = \xi_t(\Omega), \quad t \in (-\delta, \delta)
    \end{equation*}
    where $\delta > 0$ is small enough such that the flow $\xi_t$ corresponding to $V$ preserves $\pc$. In addition, we note that, up to choosing a smaller $\delta$, it holds that
    \begin{equation*}
        \xi_t(x) = x \quad \forall x \in \cone \setminus B_{\frac{3}{2}r}(x_0), \ t \in (-\delta, \delta).
    \end{equation*}

    For $t \in (-\delta, \delta)$, let $u_t \in H_0^1(\Omega_t, \cone)$ be the positive $L^2$-normalized first eigenfunction of $\Omega_t$. The proof of Proposition \ref{prop:cone_first_derivative} can be easily adapted to this general perturbation, with even more regularity (cf. \cite[Proposition 5.3.10]{HenrotPierre2018}). We obtain that the map
    \begin{equation*}
        t \in (-\delta, \delta) \mapsto u_t \in H^1(\cone)
    \end{equation*}
    is differentiable, where we abuse notation to denote an $H^1(\cone)$ extension of $u_t$ (see \cite[Section 5.3.6]{HenrotPierre2018}). We also obtain that
    \begin{equation*}
        t \in (-\delta, \delta) \mapsto |\nabla u_t|^2 \in L^1(\cone)
    \end{equation*}
    is differentiable at $t = 0$. Let us point out that the derivative
    \begin{equation*}
        \wtu = \left. \frac{d}{dt}u_t \right|_{t = 0} \in H^1(\Omega)
    \end{equation*}
    is the solution of
    \begin{equation*}
        \left\{ 
            \begin{array}{rcll}
                - \Delta \wtu & = & \lambda_1(\Omega) \wtu + \lambda' u & \quad \text{ in } \Omega \\
                \wtu & = & -\mfrac{\partial u}{\partial \nu} (V \cdot \nu) & \quad \text{ on } \Gz \\
                \mfrac{\partial \wtu}{\partial \nu} & = & 0 & \quad \text{ on } \G \\
                \displaystyle \int_\Omega \wtu u \ dx & = & 0 &
            \end{array}
        \right.
        ,
    \end{equation*} 
    where
    \begin{equation*}
        \lambda' \coloneqq \left. \frac{d}{dt} \lambda_t \right|_{t = 0}
    \end{equation*}
    By standard regularity theory, $\wtu \in W^{2, 2}(\Omega \cap B_{\frac{3}{2}r}(x_0))$ and is smooth in $\Omega$.
    
    By standard elliptic regularity results, the first eigenfunction $u$ of $\Omega$ is smooth in $\Omega \cap B_r(x_0)$. In particular,
    \begin{equation*}
        u \in W^{2, 2}(\Omega \cap B_r(x_0)),
    \end{equation*}
    and therefore
    \begin{equation*}
        |\nabla u|^2 \varphi \bar \nu \in W^{1, 1}(\cone, \R^N).
    \end{equation*}
    
    Then, we are able to apply Hadamard's formula (\cite[Theorem 5.2.2]{HenrotPierre2018}) to the function
    \begin{equation*}
        t \in (-\delta, \delta) \mapsto \lambda_t = \int_{\Omega_t}|\nabla u_t|^2 \in \R, 
    \end{equation*}
    which yields, since all objects are smooth enough for the application of Green's formula
    \begin{align}
        \lambda'
        & = 2 \int_\Omega \nabla u \nabla \wtu \ dx + \int_\Omega \divergence (|\nabla u|^2 \varphi \bar \nu) \ dx \nonumber \\
        & = 2 \underbrace{\int_\Omega \lambda_1 \wtu u \ dx}_{= 0} + 2\lambda' \underbrace{\int_\Omega u^2 \ dx}_{= 1} + \int_{\Gz \cap B_r(x_0)} |\nabla u|^2 \varphi \ d\sigma,
    \end{align}
    where $d\sigma$ is the area element of $\Gz \cap B_r(x_0)$, and the first integral is zero because of the $L^2$-orthogonality between $\wtu$ and $u$ (which comes from differentiating the normalization condition $\|u_t\|_2^2 = 1$), whence
    \begin{equation*}
        \lambda' = - \int_{\Gz \cap B_r(x_0)} |\nabla u|^2 \varphi \ d\sigma
    \end{equation*}

    On the other hand, for the function
    \begin{equation*}
        t \in (-\delta, \delta) \mapsto |\Omega_t| \in \R,
    \end{equation*}
    the standard application of Hadamard's formula yields
    \begin{equation*}
        \left. \frac{d}{dt} |\Omega_t| \right|_{t = 0} = \int_\Gz V \cdot \nu \ d\sigma = \int_{\Gz \cap B_r(x_0)} \varphi \ d\sigma.
    \end{equation*}

    Now, since $\Omega$ is a minimizer for \eqref{eq:cone_min_equiv}, it is also a minimizer for the scaling-invariant shape functional
    \begin{equation*}
        \Omega \mapsto \lambda_1(\Omega)|\Omega|^{\frac{2}{N}}.
    \end{equation*}
    Then, exploiting the optimality, we have
    \begin{align}
        0 
        & = \left. \frac{d}{dt} (\lambda_1(\Omega_t)|\Omega_t|^{\frac{2}{N}}) \right|_{t = 0} \nonumber \\
        & = |\Omega|^{\frac{2}{N}} \lambda' + \lambda_1(\Omega) \frac{2}{N} |\Omega|^{\frac{2}{N} - 1} \left(\left. \frac{d}{dt} |\Omega_t| \right|_{t = 0}\right), \nonumber
    \end{align}
    which leads to
    \begin{equation*}
        \int_{\Gz \cap B_r(x_0)} \left(|\nabla u|^2 - \frac{2 \lambda_1(\Omega)}{N |\Omega|} \right) \varphi \ d\sigma = 0.
    \end{equation*}
    Since $x_0$ and $\varphi$ are arbitrary, and $\mfrac{\partial u}{\partial \nu} < 0$ on $\Gz$ (by Hopf's Lemma), then \eqref{eq:cone_overdet_reg} follows.
\end{proof}

\begin{remark}
    As expected, the value $\sqrt{\frac{2\lambda_1(\Omega)}{N|\Omega|}}$ agrees with the Pohozaev identity (if $\Gz$ is regular).
\end{remark}

\begin{remark}
    \label{rem:cone_non-radial}
    Let $D \subset \sphere$ be a smooth domain such that $\haus(D) < \haus(\sphereplus)$ and $\mu_1 < N - 1$. Note that the minimizer $\Omega$ and the corresponding first eigenfunction $u_\Omega$ cannot be radial. Indeed: the only possible radial domains in a non-convex cone for which $\haus(\po \cap \pc) > 0$ are the intersection with an annulus or a half-sphere centred on a flat part of $\pc$. However, in the former case, the eigenfunction could be extended to the full annulus, contradicting Serrin's Theorem (\cite{Serrin1971}); in the latter case, it would contradict \eqref{eq:cone_ineq}. 
\end{remark}

Recall that $\mu_1$ is the first eigenvalue of \eqref{eq:cone_Neumann_eigenvalue_problem_D}. Then we can state the following result regarding non-radial solutions for the overdetermined problem
\begin{equation*}
    \label{eq:cone_overdet_pde_b}
    \left\{
    \begin{array}{rcll}
        - \Delta u & = & \lambda_1(\Omega) u & \quad \text{ in } \Omega \\
        u & = & 0 & \quad \text{ on } \Gz \\
        \displaystyle \mfrac{\partial u}{\partial \nu} & = & 0 & \quad \text{ on } \G \setminus \{O\} \\
        \mfrac{\partial u}{\partial \nu} & = & \text{constant} & \quad \text{ on } \Gz
    \end{array}
    \right.
\end{equation*}

\begin{theorem}
    \label{thm:cone_main}
    Let $\cone$ be the cone spanned by the smooth domain $D \subset \sphere$, with $N \geq 3$. If $\haus(D) < \haus(\sphereplus)$ and $\mu_1 < N - 1$, then there exists a non-radial open bounded domain $\Omega \subset \cone$, with an $\haus$-a.e. analytic boundary, such that 
    \eqref{eq:cone_overdet_pde} admits a non-radial solution $u_\Omega$ (with the boundary conditions satisfied $\haus$-a.e.).
\end{theorem}

\begin{proof}
    Follows by choosing $m = |\Omega_D|$ and combining Theorem \ref{thm:cone_stability}, Proposition \ref{prop:cone_measure_condition}, the results of Section \ref{sec:cone_bound_conn_reg}, and Theorem \ref{thm:cone_overdet_minimizer}, and taking Remark \ref{rem:cone_non-radial} into account.
\end{proof}

\begin{remark}[On $N = 2$]
    We note that on dimension $N = 2$, there are two conflicting possibilities for the opening angle $\theta$ (which plays the role of $\haus(D)$): either $\theta < \pi$ (in which case the cone is convex and therefore the radial sector is stable) or $\theta > \pi$, in which case the cone is non-convex but the existence of a minimizer is not guaranteed.
\end{remark}

\begin{remark}
    Let us comment on some interesting directions for further study.
    \begin{enumerate}[label=\roman*)]
        \item Uniqueness: It would be interesting to understand whether the minimizer $\Omega$ is unique or not, as well as whether it (they?) coincides with the minimizing sets for the torsion and/or isoperimetric problems studied in \cite{IacopettiPacellaWeth2022}. In my opinion, a multiplicity result for minimizing sets would be surprising (and fascinating), but it is reasonable to expect that a critical shape for one of these problems is a critical shape for the others as well.

        \item Nonlinear problems: in \cite{AfonsoIacopettiPacella2024Energypublished}, only the stability is studied, and the existence of minimizers for nonlinear problems is left open. It seems likely that our arguments could be adapted to the study of the best constant for the subcritical Sobolev inequality:
        \begin{equation*}
            C_p(\Omega) = \inf_{\substack{u \in H_0^1(\Omega \cup \Gamma) \\ \|u\|_p = 1}} \int_\Omega |\nabla u|^2 \ dx, \quad 2 < p < 2^*.
        \end{equation*}
        The stability/instability and existence of a minimizer should be very similar, while the regularity and other qualitative properties may be more challenging. It would be interesting to understand if the scheme here presented could be followed in this case, giving rise to a nonlinear, non-radial, positive solution for the overdetermined problem.

        \item Bifurcation: When $\mu_1 = N - 1$, the problem becomes degenerate, and one expects some bifurcation to occur. This has been done in cylinders (\cite{LianPacellaSicbaldi2025}) and is expected to hold also in the cone, but the geometry is more challenging.
    \end{enumerate}
\end{remark}

\section*{Acknowledgements}
This research was partially supported by Gruppo Nazionale per l'Analisi Matematica, la Probabilità e le loro Applicazioni (GNAMPA) of the Istituto Nazionale di Alta Matematica (INdAM).

\bibliographystyle{acm}
\bibliography{ref_math}

\end{document}